\documentclass[a4paper]{article}
\usepackage{mathrsfs,amsfonts,amsmath,amsthm,amssymb}
\usepackage{graphics}
\usepackage[all]{xy}
\xyoption{curve}
\xyoption{import}
\xyoption{arc}
\xyoption{ps}
\usepackage{amsmath}
\usepackage{graphicx}

\begin{document}

\numberwithin{equation}{section}

{\theoremstyle{definition}\newtheorem{definition}{Definition}[section]
\newtheorem{notation}[definition]{Notation}
\newtheorem{remnot}[definition]{Remarks and notation}
\newtheorem{terminology}[definition]{Terminology}
\newtheorem{remark}[definition]{Remark}
\newtheorem{remarks}[definition]{Remarks}
\newtheorem{example}[definition]{Example}
\newtheorem{examples}[definition]{Examples}
\newtheorem{deflem}[definition]{Definition-Lemma}
\newtheorem{proposition}[definition]{Proposition}
\newtheorem{lemma}[definition]{Lemma}
\newtheorem{theorem}[definition]{Theorem}
\newtheorem{corollary}[definition]{Corollary}

\newcommand{\ci}{C^{\infty}}
\newcommand{\A}{\mathscr{A}}
\newcommand{\Cat}{\mathscr{C}}
\newcommand{\Dnc}{\mathscr{D}}
\newcommand{\E}{\mathscr{E}}
\newcommand{\F}{\mathscr{F}}
\newcommand{\gr}{\mathscr{G}}
\newcommand{\go}{\mathscr{G} ^{(0)}}
\newcommand{\hr}{\mathscr{H}}
\newcommand{\ho}{\mathscr{H} ^{(0)}}
\newcommand{\rgr}{\mathscr{R}}
\newcommand{\rgo}{\mathscr{R} ^{(0)}}
\newcommand{\lr}{\mathscr{L}}
\newcommand{\lo}{\mathscr{} ^{(0)}}
\newcommand{\gd}{\mathscr{G}^{\mathbb{R}^2}}
\newcommand{\gt}{\mathscr{G} ^{T}}
\newcommand{\I}{\mathscr{I}}
\newcommand{\Nb}{\mathscr{N}}
\newcommand{\Kom}{\mathscr{K}}
\newcommand{\ops}{\mathscr{O}}
\newcommand{\Pb}{\mathscr{P}}
\newcommand{\sw}{\mathscr{S}}
\newcommand{\T}{\mathscr{T}}
\newcommand{\Uo}{\mathscr{U}}
\newcommand{\Vo}{\mathscr{V}}
\newcommand{\Wo}{\mathscr{W}}
\newcommand{\Rr}{\mathbb{R}}
\newcommand{\Nat}{\mathbb{N}}
\newcommand{\C}{\mathbb{C}}
\newcommand{\src}{\mathscr{S}_{c}}
\newcommand{\cc}{C_{c}^{\infty}}
\newcommand{\cg}{C_{c}^{\infty}(\gr)}
\newcommand{\cgo}{C_{c}^{\infty}(\go)}
\newcommand{\ct}{C_{c}^{\infty}(\gr^T)}
\newcommand{\ca}{C_{c}^{\infty}(A\gr)}
\newcommand{\Un}{{U}^{(n)}}
\newcommand{\Du}{D_{\mathscr{U}}}
\newcommand{\cT}{\mathcal{T}}

\renewcommand{\o}{ \mathfrak {o }}

  \renewcommand{\a}{\alpha}
  \renewcommand{\b}{\beta}
   \newcommand{\io}{\iota}
  
  \newcommand{\eps}{\epsilon}
  \renewcommand{\d}{\delta}
  \newcommand{\pa}{\partial}

  \newcommand{\ind}{{\bf  Index }}

  \newcommand{\CC}{{\mathbb C}}
  \newcommand{\RR}{{\mathbb R}}
  \newcommand{\ZZ}{{\mathbb Z}}
  \newcommand{\QQ}{{\mathbb Q}}
  \newcommand{\NN}{{\mathbb  N}}
    \newcommand{\KK}{{\mathbb  K}}

  \newcommand{\PP}{{\mathbb P}}

\def\to{\longrightarrow}

\def\L{\mathop{\wedge}}

\def\gpd{\,\lower1pt\hbox{$\longrightarrow$}\hskip-.24in\raise2pt
             \hbox{$\longrightarrow$}\,}

\begin{center}
{\Large\bf Geometric Baum-Connes assembly map for twisted Differentiable Stacks
}

\bigskip

{\sc by Paulo Carrillo Rouse and Bai-Ling Wang}

\end{center}

{\footnotesize
\vskip 2pt Institut de Math\'ematiques de Toulouse
\vskip -4pt Universit\'e de Toulouse
\vskip -4pt F-31062 Toulouse cedex 9, France
\vskip -4pt e-mail: paulo.carrillo@math.univ-toulouse.fr

\vskip 2ptDepartment of Mathematics
\vskip-4pt Mathematical Sciences Institute
\vskip-4ptAustralian National University
\vskip-4pt Canberra, ACT 2600, Australia.
\vskip -4pt e-mail: bai-ling.wang@anu.edu.au
}
\bigskip
\everymath={\displaystyle}

\begin{abstract}
\noindent 
We construct the geometric Baum-Connes assembly map for twisted Lie groupoids, that means for Lie groupoids together with a given groupoid equivariant $PU(H)-$principle bundle. The construction  is based on the use of geometric deformation groupoids, these objects allow in particular to give a geometric construction of the associated pushforward maps and to  establish  the functoriality. The main results in this paper are  to define the geometric twisted K-homology groups and to construct the assembly map.
Even in the untwisted case the fact that the geometric  K-homology groups and the geometric assembly map are well defined for Lie groupoids is new, as it was only sketched by Connes in his book for general Lie groupoids without any restrictive hypothesis, in particular for non Hausdorff Lie groupoids.

We also prove the Morita invariance of the assembly map, giving thus a precise meaning to the geometric assembly map for twisted differentiable stacks. We discuss the relation of the assembly map with the associated assembly map of the $S^1$-central extension. The relation with the analytic assembly map is treated, as well as some cases in which we have an isomorphism. One important tool is the twisted Thom isomorphism in the groupoid equivariant case which we establish in the appendix.

\begin{center}
{\bf R\'esum\'e}
\end{center}
Nous construisons le morphisme d'assemblage g\'eom\'etrique de Baum-Connes pour des groupo\"ides de Lie tordus, \`a savoir des groupo\"ides de Lie avec un $PU(H)$-fibr\'e principal equivariant. La construction est bas\'e dans l'utilisation des groupo\"ides de d\'eformation, ces objets permettent en particulier de donner une construction g\'eom\'etrique des morphismes shriek associ\'es et d'\'etablir la fonctorialit\'e. Les r\'esultats principaux de cet article sont la d\'efinition des groupes de K-homologie g\'eom\'etrique tordue et la construction du morphisme d'assemblage.
M\^eme dans le cas non tordu le fait que les groupes de K-homologie g\'eom\'etrique et le morphisme d'assemblage (g\'eom\'etrique) pour des groupo\"ides de Lie sont bien d\'efinis est nouveau, en effet, ceci a \'et\'e esquiss\'e par Connes dans son livre pour des groupo\"ides de Lie g\'en\'erales sans aucune restriction, en particulier pour des groupo\"ides non s\'epar\'es.

Nous montrons aussi l'invariance par Morita du morphisme d'assemblage, donnant ainsi un sens pr\'ecis au morphisme d'assemblage g\'eom\'etrique de Baum-Connes pour des Champs diff\'erentiables tordus. Nous discutons la relation de notre morphisme d'assemblage avec le morphisme associ\'e \`a la $S^1$-extension central. La relation avec le morphisme analytique est trait\'e, ainsi que quelques cas o\`u il y a isomorphisme. Un outil important est le morphisme de Thom tordu dans le cas equivariant par rapport \`a un groupo\"ide que nous \'etablissons dans l'appendice.

\end{abstract}

Keywords: Twisted K-theory, Index theory, Lie groupoids and Differentiable stacks. MSC 58J22 (19K35,19K56,46L80)
Mots-cl\'es: K-th\'eorie tordue, Th\'eorie de l'indice, Groupo\"ides de Lie et Champs diff\'erentiables. 

\tableofcontents

\section{Introduction}
The present paper is a natural sequel of \cite{CaWangAdv} where we started a study of an index theory for foliations with the presence of $PU(H)$-twistings (see also \cite{CaWangCRAS}).

In \cite{BC} Baum and Connes introduced a geometrically defined K-theory for Lie groups, group actions and foliations. Its main features are its computability and simplicity of its definition, besides, in some cases they were able to construct a (also geometric) Chern character. Using classic ideas from index theory they constructed a natural map from this group to the analytic K-theory. This so-called Baum-Connes assembly map gave rise to many research developments due to its connection to many areas of mathematics and mathematical physics. Very interesting geometric and analytic corollaries can be deduced from the injectivity, surjectivity or bijectivity of the Baum-Connes map. Shortly after the paper by Baum-Connes, the powerful tools of KK-theory took over the originally geometrically defined map. Indeed, the use of KK-theory to define the assembly map have given extraordinary results. However the original geometrically defined map was somehow lost. In fact for some years experts assume both approaches to be the same but it took some years to give the actual proof for some cases.

The geometric approach is very interesting for several reasons, for instance the use of geometric K-homology in index theory and hence a completely geometric way of doing index theory, the possibility of defining a (geometric) Chern character from the geometric K-homology and hence to obtain explicit formulae. It is more suitable for geometric situation for which the analytic approach is not yet understood, for example for general Lie groupoids the analytic assembly is only defined for Hausdorff groupoids.  

In this paper we construct the geometric Baum-Connes map for a twisted Lie groupoid $(\gr, \alpha)$, that is a  Lie groupoid
$\gr$  together with a given equivariant (with respect to the groupoid action) $PU(H)-$principle bundle on $\gr$. Equivalently,  a twisting is given by a Hilsum-Skandalis morphism
\[
  \xymatrix{\alpha:\, \gr \ar@{-->}[r] &PU(H)}.
\]
  Even in the untwisted case this was not done before. In fact, in Connes book (\cite{Concg} II.10.$\alpha$), he proposes a definition for the geometric group of a Lie groupoid and he sketches the construction for the assembly map using deformation groupoids ideas that englobes what he did in \cite{BC} with Baum. We utilize  these ideas  to
study the assembly map for the twisted case.

Let $\gr\rightrightarrows M$ be a Lie groupoid with a given twisting $\alpha$ on $\gr$. Given such a data we can consider the maximal $C^*$-algebra $C^*(\gr,\alpha)$ (or reduced if indicated), the algebra is constructed by taking a $S^1$-central extension $R_\alpha$ associated to $\alpha$ via the canonical $S^1$-central extension $S^1\to U(H)\to PU(H)$ and using one factor of the algebra associated to such extension\footnote{The extension depends of the choice of a cocycle defining $\alpha$, however two such extensions are Morita equivalent via an explicit equivalence and hence the algebras they define are Morita  equivalent as well.}, for complete details see section 3.1 below. 

Now, consider a $\gr$-manifold  $P$ with momentum map   $\pi_P:P\to M$  which is assumed to be a submersion. Denote by $T^vP$ the vertical tangent bundle associated to $\pi_P$. In this paper we will assume that for any $\gr$-manifold $P$, $T^vP$ is an oriented vector bundle which admits a $\gr$-invariant metric, for instance when $\gr$ acts on $P$ properly or when $P=M$. We will denote the twisted analytic K-theory groups of the action groupoid $P\rtimes \gr$ by
$$K^*(P\rtimes \gr,\alpha):=K_{-*}(C^{*}(P\rtimes \gr,\pi_P^*\alpha))$$
where $\pi_P^*\alpha$ is the  pull-back twisting on $P\rtimes \gr$  by  the groupoid morphism $\pi_P:P\rtimes \gr \to \gr$ (we use the same notation for $\pi_P$ at the level of the arrows). One can consider a $S^1$-central extension $R_\alpha$ over a Cech groupoid $\gr_\Omega$ (Morita equivalent to $\gr$).  If there is an extra twisting we will add it in the notation and explain it case by case.

Let $P,N$ be two $\gr$-manifolds and $f:P\longrightarrow N$ a $\gr$-equivariant oriented smooth map. Using only geometric deformation groupoids, we  construct a morphism\footnote{In \cite{CaWangAdv} the special case where $\gr$ is the holonomy groupoid of a foliation and the action on $P$ is free is treated, we proved there in particular the functoriality as an application of a longitudinal index theorem.} , the shriek map, 
\begin{equation}\label{fshriekintrosection0}
\xymatrix{
K^*(P\rtimes \gr,\alpha+\o_f)\ar[r]^-{f_!}&K^*(N\rtimes \gr,\alpha)
}
\end{equation}
where $\o_f$ is the  orientation twisting\footnote{in the sense of example \ref{obundle} in \ref{example}.} over $P\rtimes \gr$ of   the 
$\gr$-vector bundle $f^*T^vN\oplus T^vP$.

We remark  that the construction of the shriek map is by means of deformation groupoids, this gives a explicit geometric pushforward map that gives exactly the corresponding  equivariant family index when $f$ is  a submersion. Moreover, we establish the functoriality of the construction by again only using deformation groupoids, this gives a very geometric flavour to the proof, indeed one can understand the functoriality via a double deformation from one groupoid to another one. As we mentioned above, this was not done before even in the untwisted case, in fact, in \cite{Concg} (section II.6) Connes sketched the construction for the classic pushforward between manifolds using deformation groupoids and left the proof of  the functoriality as an exercise. We remark  that the result below (theorem \ref{twistfun}) was proved (for $f, g$ submersions) using analytic methods by Tu and Xu (\cite{TXring} 4.19), the statement is the following:

\begin{theorem}
The push-forward morphism (\ref{fshriekintrosection0})  is functorial, that means, if we have a composition of smooth $\gr$-oriented smooth maps between two 
$\gr-$manifolds 
$
P\stackrel{f}{\longrightarrow}N\stackrel{g}{\longrightarrow}L,
$
and a twisting  $\xymatrix{\alpha:\gr \ar@{-->}[r] &PU(H)}$,  then the following diagram commutes
\[
\xymatrix{
K^*(P\rtimes \gr,\alpha+\o_{g\circ f}) \ar[rr]^-{(g\circ f)_!}\ar[rd]_-{f_!}&&K^*(L\rtimes \gr,\alpha)\\
&K^*(N\rtimes \gr,\alpha+\o_g)\ar[ru]_-{g_!}&
}
\]

\end{theorem}

The above theorem enables us to define the associated geometric K-homology group for a  Lie groupoid with a twisting. 

\begin{definition}[Twisted geometric K-homology]
Let $\gr\rightrightarrows M$ be a Lie groupoid with a twisting 
$\alpha:\gr--->PU(H)$. The twisted geometric K-homology group  associated to 
$(\gr,\alpha)$ is the abelian group denoted by $K_{*}^{geo}(\gr,\alpha)$ with generators and relations described as follows. A generator is called a   cycle  $(P,\xi)$ where
\begin{itemize}
\item[(1)]   $P$ is a smooth  co-compact $\gr$-proper manifold,
\item[(2)]   $\pi_P:P\to M$ is  the smooth momentum map which is supposed to be an oriented submersion, and
\item[(3)] $\xi \in  K^{*}(P\rtimes \gr, \pi_P^*\alpha+\o_{T^vP})$,
\end{itemize}
and two cycles $(P, \xi)$ and $(P, \xi')$ are called equivalent if there is  a 
  smooth $\gr$-equivariant map  $g:P\to P'$  such that 
\begin{equation}
\xi' = g_! (\xi).
\end{equation}
\end{definition}

One of the reasons for calling this group "geometric" is that the groupoid $P\rtimes \gr$ is proper and hence its twisted K-theory can be expressed in good cases by twisted vector bundles (\cite{TXL} theorem 5.28). Another important reason is that from the twisted K-theory for proper \'etale groupoids Tu and Xu constructed the Baum-Connes delocalized Chern character with values in the twisted cohomology of the associated inertia groupoid, they prove that their Chern character gives a rational isomorphism, \cite{TXChern}.
We will come to this discussion later. For the moment let us mention that we can perform some basic computations, see Example
\ref{exfinalobj}. 

Now we summarize Theorems \ref{MoritageoK}, \ref{twistedassemblymap} and \ref{MoritaAS}  in this paper as  follows.

 \begin{theorem}
 Let $(P,x)$ be a geometric cycle over $(\gr,\alpha)$. Let $\mu_{\alpha}(P,x)=(\pi_P)_!(x)$ be the element   in $K^*(\gr,\alpha)$. Then
  $\mu_{\alpha}(P,x)$ only depends upon the equivalence class of the twisted cycle  $(P,x)$. Hence we have a well defined assembly map
\begin{equation}
\mu_{\alpha}:K_{*}^{geo}(\gr,\alpha)\longrightarrow K^*(\gr,\alpha). 
\end{equation}
Moreover,  the assembly map satisfies   the Morita invariance  in the following sense: 
Let $\gr$ and $\gr'$ are two Morita equivalent groupoids. Let us denote by $\gr\stackrel{\phi}{--->}\gr'$ the generalized isomorphism (the Morita bi-bundle). Given  a twisting  $\alpha':\gr'---> PU(H)$, there is a commutative diagram
\begin{equation}
\xymatrix{
K^{geo}(\gr,\alpha)\ar[r]^-{\phi_*}_-{\cong}\ar[d]_-{\mu_\alpha}&K^{geo}(\gr',\alpha')\ar[d]^-{\mu_\alpha'}\\
K^*(\gr,\alpha)\ar[r]_-{\phi_*}^-{\cong}&K^*(\gr',\alpha')
}
\end{equation}
where $\alpha:=\alpha'\circ\phi$ is the induced twisting on $\gr$.
\end{theorem}
 
\begin{remark}  
The Morita invariance of the assembly map\footnote{The Morita invariance of the geometric assembly map is proven for the untwisted case in \cite{Shim}, but in that paper the author did not discuss  that the assembly map is well defined.} is important in many applications. It justifies in one hand the fact that the construction does not depend on the given cocycle representing the twisting neither on the given associated extension (modulo an explicit induced Morita isomorphism).  more importantly it gives a precise  meaning to the twisted assembly map for differentiable stacks. This last point is essential since in practice one usually changes the groupoid model by a Morita equivalent one (for some examples on Morita equivalences see section \ref{HScat} below).
\end{remark}

 For the case of a proper groupoid $\gr\rightrightarrows M$ with $M/\gr$ compact,   the assembly map is an isomorphism  (Cf. Proposition \ref{propQfinal}). This covers the case of compact orbifold groupoids. 
 For   a connected Lie group $G$ with  a projective representation  $\alpha: G\to PU(H)$,   let $L$ be a maximal compact subgroup of $G$. Then  we have a commutative diagram
\begin{equation}
\xymatrix{
K_{*}^{geo}(G,\alpha)\ar[rd]_-{\mu_\alpha}\ar[rr]^-{\mu_L}_-{\cong}&& K^*(L, i^*\alpha+\o_{T_e(L\setminus G)})\ar[ld]^-{i_!}\\
&K^*(G,\alpha)&
}
\end{equation}
where $i:L\hookrightarrow G$ is the restriction morphism. In the case $\alpha$ and $\o_{T_e(L\setminus G)}$ are trivial, the above diagram gives a meaning to Mackey's observations on unitary representations for Lie groups, at least in the case where the assembly map is an isomorphism. For an almost connected Lie group $G$, this is  known as the Connes-Kasparov conjecture proved in \cite{CEN}.
In the twisted case there should also be a relation between the projective representations of some Lie groups and certain related semi-direct product group's projective representations\footnote{By Thom isomorphism $K^*(L, i^*\alpha+\o_{T_e(L\setminus G)})\cong K^*(T_e(L\setminus G)\rtimes L,i^*\alpha)$}. This will be discussed elsewhere.
 

Next, we discuss the relation of the assembly maps with the associated assembly map for the groupoid extension. This gives a precise meaning to the twisted assembly as the degree one part of a classic assembly map under the $S^1$-action. More explicitely,  given an extension groupoid $R_{\alpha}$ associated to $(\gr,\alpha)$, the $S^1$-action on $R_{\alpha}$ induces a $\mathbb{Z}$-grading  in $C^*(R_{\alpha})$ (Proposition 3.2 in \cite{TXL}). We have
$$K^*(R_\alpha)\cong \bigoplus_{n\in \mathbb{Z}}K^*(\gr,n\alpha).$$
Now, for the Lie groupoid $R_\alpha$ there is a geometric assembly map $\mu_{R_\alpha}$. The following results (See Proposition \ref{gradprop}) relates the assembly map $\mu_{R_\alpha}$  with the assembly map for the twisted Lie  groupoid. 

\begin{proposition}
We have an isomorphism of groups
\begin{equation}
K^{geo}_{*}(R_\alpha)
 \cong \bigoplus_{n\in \mathbb{Z}}K^{geo}_*(\gr,n\alpha)  \end{equation}
and under this isomorphism
$
 \mu_{R_\alpha} = \bigoplus_{n\in \mathbb{Z}}\mu_{n\alpha} . 
$   In particular the geometric twisted assembly map is an isomorphism whenever the geometric assembly map for the corresponding  extension is.
\end{proposition}

{\bf Comparison with the analytic assembly:}

Up to now, we have not supposed our groupoids to be Hausdorff. In the Hausdorff case there is an analytic version of the assembly map that has been widely studied, in particular thanks to  Kasparov's KK-theory. In this case, we have the following comparison result (Cf. Proposition \ref{assemblytopana}):

\begin{proposition} Let $R$ be a Hausdorff groupoid. There exists a homomorphism $\lambda_R: K^{geo}_{*}(R)\to K^{ana}_{*}(R)$ such that, denoting by $\mu_{R}^{ana}$ the analytic assembly map (\cite{Tu3}), the following diagram commutes 
\begin{equation}
\xymatrix{
K^{geo}_{*}(R)\ar[rd]_-{\mu_{R}}\ar[rr]^-{\lambda_R}&&K^{ana}_{*}(R)
\ar[ld]^-{\mu_{R}^{ana}}\\
&K^*(R)&
}
\end{equation}
Moreover,  the Morita invariance of each morphism in the above commutative diagram holds.
\end{proposition}


In the case of Hausdorff groupoids, assuming $\lambda_{R_\alpha}:K^{geo}_{*}(R_\alpha)\longrightarrow K^{ana}_{*}(R_\alpha)$ is an isomorphism, then  the geometric twisted assembly map for $(\gr,\alpha)$ is an isomorphism whenever the analytic assembly map for $R_\alpha$ is.  An  interesting example of the this  situation is when the groupoid 
$\gr$ satisfies the so called Haagerup property. Indeed, in this case, one can check that for any twisting $\alpha$, the correspondent extension groupoid $R_\alpha$ satisfies as well the Haagerup property. Then by Tu's theorem (\cite{Tu2} theorem 9.3, see also \cite{Tu3} theorem 6.1) the analytic assembly map for $R_\alpha$ is an isomorphism. This was already mentioned in Tu's habilitation \cite{Tu4} page 16.
Some examples of Lie groupoids for which the (reduced, see remark below) analytic assembly map is known to be an isomorphism  or injective  are 
\begin{enumerate}
\item injectivity for bolic groupoids (Tu \cite{Tu1}),
\item isomorphism for groupoids having the Haagerup property (Tu \cite{Tu2}),
\item isomorphism for almost connected Lie groups (Chabert-Echterhoff-Nest \cite{CEN}),
\item isomorphism for hyperbolic groups (Lafforgue \cite{Laff12}).
\end{enumerate}
A very interesting question then is the following one:

{\bf Question:} For which Lie groupoids is the comparison map $\lambda$ an isomorphism?

Let us mention that different models for K-homology (at least in the untwisted case) were assumed by the experts to be isomorphic for many years. It was not until some years ago that a complete proof for some models was
provided (\cite{BHS,BOOSW}). So the above  question  is  far from be trivial and as we stated above a positive answer would have   some interesting applications. 
In this paper we have only discussed two models for twisted K-homology, but, as we indicate in \cite{CaWangCRAS} for foliations, there is also a Baum-Douglas geometric model for twisted Lie groupoids (See \cite{W08} where the second author introduced the case for twisted manifolds). The Baum-Douglas geometric model is easily seen to be isomorphic to the geometric one proposed here and it has the advantage that similar methods as in \cite{BHS,BOOSW} apply for a very large family of Lie groupoids. We will discuss this   in a forthcoming paper.

\begin{remark}[About the use of maximal or reduced $C^*$-algebras]
The reduced $C^*$-algebra is in principle more geometrical. For instance, the twisted $K$-theory can be described in some cases by twisted vector bundles, theorem 5.28 in \cite{TXL}. For some groupoids (amenable, K-amenable, etc...) the reduced and the maximal completions coincide. For example, in the definition of the geometric K-homology group above, one has cycles in $K^{*}(C^{*}_{red}(P\rtimes \gr, \pi_P^*\alpha+\o_{T^vP}))=K^{*}(C^{*}(P\rtimes \gr, \pi_P^*\alpha+\o_{T^vP}))$ since $P\rtimes \gr$ is proper.
 By taking the canonical induced morphism from the K-theory of a maximal $C^*$-algebra to the K-theory of the reduced one, we can define the assembly map with values in the K-theory of the reduced $C^*$-algebra of a twisted groupoid. All the results above concerning the assembly map still hold for the "reduced" assembly map. 
 
  The problem in adapting directly our results to the reduced case is a problem of exactness. In his thesis \cite{Las}, Lassagne studies under which conditions the pushforward maps between foliation groupoids can be performed directly in the reduced $C^*$-algebra level. Another possibility is to adapt to groupoids the recent reformulated Baum-Connes conjecture proposed by Baum-Guenter-Willett in \cite{BGW}, there the authors define a minimal (Morita invariant) crossed product for which on does not have anymore the exactness problems mentioned above. One can certainly define in this context the reformulated twisted Baum-Connes assembly map.
\end{remark}


{\bf Acknowledgements:} We would like to thank the referee for carefully reading our work and for making important remarks on the twisted Thom isomorphism that led us to a net improvement of the paper.
The first author would like to thank the excellent working conditions he had at the Max Planck Institut for Mathematics at Bonn where part of this work was realized. 

\section{Preliminaries on groupoids}

In this section, we review the notion of twistings on Lie groupoids and 
discuss some examples which appear in this paper.
Let us recall what a groupoid is:

\begin{definition}
A $\it{groupoid}$ consists of the following data:
two sets $\gr$ and $\go$, and maps
\begin{itemize}
\item[(1)]  $s,r:\gr \rightarrow \go$ 
called the source map and target map respectively,
\item[(2)]  $m:\gr^{(2)}\rightarrow \gr$ called the product map 
(where $\gr^{(2)}=\{ (\gamma,\eta)\in \gr \times \gr : s(\gamma)=r(\eta)\}$),
\end{itemize}
together with  two additional  maps, $u:\go \rightarrow \gr$ (the unit map) and 
$i:\gr \rightarrow \gr$ (the inverse map),
such that, if we denote $m(\gamma,\eta)=\gamma \cdot \eta$, $u(x)=x$ and 
$i(\gamma)=\gamma^{-1}$, we have 
\begin{itemize}
\item[(i)] $r(\gamma \cdot \eta) =r(\gamma)$ and $s(\gamma \cdot \eta) =s(\eta)$.
\item[(ii)] $\gamma \cdot (\eta \cdot \delta)=(\gamma \cdot \eta )\cdot \delta$, 
$\forall \gamma,\eta,\delta \in \gr$ whenever this makes sense.
\item[(iii)] $\gamma \cdot x = \gamma$ and $x\cdot \eta =\eta$, $\forall
  \gamma,\eta \in \gr$ with $s(\gamma)=x$ and $r(\eta)=x$.
\item[(iv)] $\gamma \cdot \gamma^{-1} =u(r(\gamma))$ and 
$\gamma^{-1} \cdot \gamma =u(s(\gamma))$, $\forall \gamma \in \gr$.
\end{itemize}
For simplicity, we denote a groupoid by $\gr \rightrightarrows \go $. A strict morphism $f$ from
a  groupoid   $\hr \rightrightarrows \ho $  to a groupoid   $\gr \rightrightarrows \go $ is  given
by  maps 
\[
\xymatrix{
\hr \ar@<.5ex>[d]\ar@<-.5ex>[d] \ar[r]^f& \gr \ar@<.5ex>[d]\ar@<-.5ex>[d]\\
\ho\ar[r]_{f_0}&\go
}
\]
which preserve the groupoid structure, i.e.,  $f$ commutes with the source, target, unit, inverse  maps, and respects the groupoid product  in the sense that $f(h_1\cdot h_2) = f (h_1) \cdot f(h_2)$ for any $(h_1, h_2) \in \hr^{(2)}$.

\end{definition}

In  this paper we will only deal with Lie groupoids, that is, 
a groupoid in which $\gr$ and $\go$ are smooth manifolds, and $s,r,m,u$ are smooth maps (with s and r submersions, see \cite{Mac,Pat}). 

\subsection{The tangent groupoid}

In this subsection, we review the notion of Connes' tangent groupoids from deformation to the normal cone point of view.

\subsubsection{Deformation to the normal cone}\label{DCN}

The tangent groupoid is a particular case of a geometric construction that we describe here.

Let $M$ be a $\ci$ manifold and $X\subset M$ be a $\ci$ submanifold. We denote
by $\mathscr{N}_{X}^{M}$ the normal bundle to $X$ in $M$.
We define the following set
\begin{align}
\mathscr{D}_{X}^{M}:= \left( \mathscr{N}_{X}^{M} \times {0} \right) \bigsqcup   \left(M \times \mathbb{R}^* \right). 
\end{align} 
The purpose of this section is to recall how to define a $\ci$-structure in $\mathscr{D}_{X}^{M}$. This is more or less classical, for example
it was extensively used in \cite{HS}.

Let us first consider the case where $M=\mathbb{R}^p\times \mathbb{R}^q$ 
and $X=\mathbb{R}^p \times \{ 0\}$ ( here we
identify  $X$ canonically with $ \mathbb{R}^p$). We denote by
$q=n-p$ and by $\mathscr{D}_{p}^{n}$ for $\mathscr{D}_{\mathbb{R}^p}^{\mathbb{R}^n}$ as above. In this case
we   have that $\mathscr{D}_{p}^{n}=\mathbb{R}^p \times \mathbb{R}^q \times \mathbb{R}$ (as a
set). Consider the 
bijection  $\psi: \mathbb{R}^p \times \mathbb{R}^q \times \mathbb{R} \rightarrow
\mathscr{D}_{p}^{n}$ given by 
\begin{equation}\label{psi}
\psi(x,\xi ,t) = \left\{ 
\begin{array}{cc}
(x,\xi ,0) &\mbox{ if } t=0 \\
(x,t\xi ,t) &\mbox{ if } t\neq0
\end{array}\right.
\end{equation}
whose  inverse is given explicitly by 
$$
\psi^{-1}(x,\xi ,t) = \left\{ 
\begin{array}{cc}
(x,\xi ,0) &\mbox{ if } t=0 \\
(x,\frac{1}{t}\xi ,t) &\mbox{ if } t\neq0
\end{array}\right.
$$
We can consider the $\ci$-structure on $\mathscr{D}_{p}^{n}$
induced by this bijection.

We pass now to the general case. A local chart 
$(\mathscr{U},\phi)$ of $M$ at $x$  is said to be a $X$-slice   if 
\begin{itemize}
\item[1)]  $\mathscr{U}$  is an open neighbourhood of $x$ in $M$ and  $\phi : \mathscr{U}  \rightarrow U \subset \mathbb{R}^p\times \mathbb{R}^q$ is a diffeomorphsim such that $\phi(x) =(0, 0)$. 
\item[2)]  Setting $V =U \cap (\mathbb{R}^p \times \{ 0\})$, then
$\phi^{-1}(V) =   \mathscr{U} \cap X$ , denoted by $\mathscr{V}$.
\end{itemize}
With these notations understood, we have $\mathscr{D}_{V}^{U}\subset \mathscr{D}_{p}^{n}$ as an
open subset.   For $x\in \mathscr{V}$ we have $\phi (x)\in \mathbb{R}^p
\times \{0\}$. If we write 
$\phi(x)=(\phi_1(x),0)$, then 
$$ \phi_1 :\mathscr{V} \rightarrow V \subset \mathbb{R}^p$$ 
is a diffeomorphism.  Define a function 
\begin{equation}\label{phi}
\tilde{\phi}:\mathscr{D}_{\mathscr{V}}^{\mathscr{U}} \rightarrow \mathscr{D}_{V}^{U} 
\end{equation}
by setting 
$\tilde{\phi}(v,\xi ,0)= (\phi_1 (v),d_N\phi_v (\xi ),0)$ and 
$\tilde{\phi}(u,t)= (\phi (u),t)$ 
for $t\neq 0$. Here 
$d_N\phi_v: N_v \rightarrow \mathbb{R}^q$ is the normal component of the
 derivative $d\phi_v$ for $v\in \mathscr{V}$. It is clear that $\tilde{\phi}$ is
 also a  bijection. In particular,  it induces a $C^{\infty}$ structure on $\mathscr{D}_{\mathscr{V}}^{\mathscr{U}}$. 
Now, let us consider an atlas 
$ \{ (\mathscr{U}_{\alpha},\phi_{\alpha}) \}_{\alpha \in \Delta}$ of $M$
 consisting of $X-$slices. Then the collection $ \{ (\mathscr{D}_{\mathscr{V}_{\alpha}}^{\mathscr{U}_{\alpha}},\tilde{\phi}_{\alpha})
  \} _{\alpha \in \Delta }$ is a $\ci$-atlas of
  $\mathscr{D}_{X}^{M}$ (Proposition 3.1 in \cite{Ca4}).

\begin{definition}[Deformation to the normal cone]
Let $X\subset M$ be as above. The set
$\mathscr{D}_{X}^{M}$ equipped with the  $C^{\infty}$ structure
induced by the atlas of  $X$-slices is called
 the deformation to the  normal cone associated  to   the embedding
$X\subset M$. 
\end{definition}


One important feature about the deformation to the normal cone is the functoriality. More explicitly,  let
 $f:(M,X)\rightarrow (M',X')$
be a   $\ci$ map   
$f:M\rightarrow M'$  with $f(X)\subset X'$. Define 
$ \mathscr{D}(f): \mathscr{D}_{X}^{M} \rightarrow \mathscr{D}_{X'}^{M'} $ by the following formulas: \begin{enumerate}
\item[1)] $\mathscr{D}(f) (m ,t)= (f(m),t)$ for $t\neq 0$, 

\item[2)]  $\mathscr{D}(f) (x,\xi ,0)= (f(x),d_Nf_x (\xi),0)$,
where $d_Nf_x$ is by definition the map
\[  (\mathscr{N}_{X}^{M})_x 
\stackrel{d_Nf_x}{\longrightarrow}  (\mathscr{N}_{X'}^{M'})_{f(x)} \]
induced by $ T_xM 
\stackrel{df_x}{\longrightarrow}  T_{f(x)}M'$.
\end{enumerate}
 Then $\mathscr{D}(f):\mathscr{D}_{X}^{M} \rightarrow \mathscr{D}_{X'}^{M'}$ is a $\ci$-map (Proposition 3.4 in \cite{Ca4}). In the language of categories, the deformation to the normal cone  construction defines a functor
\begin{equation}\label{fundnc}
\mathscr{D}: \mathscr{C}_2^{\infty}\longrightarrow \mathscr{C}^{\infty} ,
\end{equation}
where $\mathscr{C}^{\infty}$ is the category of $\ci$-manifolds and $\mathscr{C}^{\infty}_2$ is the category of pairs of $\ci$-manifolds.

\begin{definition}[Tangent groupoid]
Let $\gr \rightrightarrows \go $ be a Lie groupoid. $\it{The\, tangent\,
groupoid}$ associated to $\gr$ is the groupoid that has 
\[
\mathscr{D}_{\go}^{\gr} = \left( \mathscr{N}_{\go}^{\gr} \times \{0\}\right) \bigsqcup  \left( \gr\times \mathbb{R}^*\right)
\]
 as the set of arrows and  $\go \times \mathbb{R}$ as the units, with:
 \begin{enumerate}
\item  $s^T(x,\eta ,0) =(x,0)$ and $r^T(x,\eta ,0) =(x,0)$ at $t=0$.
\item   $s^T(\gamma,t) =(s(\gamma),t)$ and $r^T(\gamma,t)
  =(r(\gamma),t)$ at $t\neq0$.
\item  The product is given by
  $m^T((x,\eta,0),(x,\xi,0))=(x,\eta +\xi ,0)$ and  $m^T((\gamma,t), 
  (\beta ,t))= (m(\gamma,\beta) , t)$ if $t\neq 0 $ and 
if $r(\beta)=s(\gamma)$.
\item The unit map $u^T:\go \rightarrow \gr^T$ is given by
 $u^T(x,0)=(x,0)$ and $u^T(x,t)=(u(x),t)$ for $t\neq 0$.
\end{enumerate}

We denote $\gr^{T}= \mathscr{D}_{\go}^{\gr}$ and $A\gr =\mathscr{N}_{\go}^{\gr}  $ as a vector bundle over $\gr^{(0)}$. Then we have a family of Lie groupoids parametrized by $\mathbb{R}$, which itself is a Lie groupoid
\[
\gr^T = \left( A\gr \times \{0\} \right)   \bigsqcup \left(  \gr\times   \mathbb{R}^*  \right) \rightrightarrows \go\times \mathbb{R}.
\]
\end{definition} 
As a consequence of the functoriality of the deformation to the normal cone,
one can show that the tangent groupoid is in fact a Lie
groupoid compatible with the Lie groupoid structures of $\gr$ and $A\gr$. Here $A\gr  \rightrightarrows \go$ is  considered as a Lie groupoid  defined by  the  vector bundle structure. 
Indeed, it is immediate that if we identify in a
canonical way $\mathscr{D}_{\go}^{\gr^{(2)}}$ with $(\gr^T)^{(2)}$, then 
$$ m^T=\mathscr{D}(m),\, s^T=\mathscr{D}(s), \,  r^T=\mathscr{D}(r),\,  u^T=\mathscr{D}(u)$$
where we are considering the following  smooth maps of pairs:
\[\begin{array}{l}
m:( \gr^{(2)},\go)\rightarrow (\gr,\go ), \nonumber
\\
s,r:(\gr ,\go) \rightarrow (\go,\go),\nonumber 
\\
u:(\go,\go)\rightarrow (\gr,\go ).\nonumber
\end{array}
\]

\subsection{The Hilsum-Skandalis category}\label{HScat}
 
Lie groupoids form a category with  strict  morphisms of groupoids. It is now a well-established fact  in Lie groupoid's theory that the right category to consider is the one in which Morita equivalences correspond precisely to isomorphisms.  We review some basic definitions and properties of generalized morphisms between Lie groupoids, see \cite{TXL} section 2.1, or 
\cite{HS,Mr,MM} for more detailed discussions.

\begin{definition}[Generalized homomorphisms]\label{HSmorphism}   
Let $\gr \rightrightarrows \go$ and  
$\hr \rightrightarrows \ho$ be two Lie groupoids.  A generalized groupoid morphism, also called a Hilsum-Skandalis morphism, from $\hr$ to $\gr$ is given by  principal $\gr$-bundle over $\hr$, that 
is, a right  principal $\gr$-bundle over $\ho$
which is also a left $\hr$-bundle over $\go$ such that the   the right $\gr$-action and the left 
$\hr$-action commute,  formally denoted by
\[
f:  \xymatrix{\hr \ar@{-->}[r] &  \gr}
\]
or by  
\[
\xymatrix{
\hr \ar@<.5ex>[d]\ar@<-.5ex>[d]&P_f \ar@{->>}[ld] \ar[rd]&\gr \ar@<.5ex>[d]\ar@<-.5ex>[d]\\
\ho&&\go.
}
\]
if   we want to emphsize  the bi-bundle $P_f$ involved. 
\end{definition}

Notice that a generalized morphism (or Hilsum-Skandalis morphism),   
$f:  \xymatrix{\hr \ar@{-->}[r] &  \gr}$, is given by one of the three equivalent data:
\begin{enumerate}
\item A  locally trivial  right  principal $\gr$-bundle $P_f$  over  $\hr$  as Definition \ref{HSmorphism}. 
\item A 1-cocycle $f=\{(\Omega_i,f_{ij})\}_{i\in I}$ on $\hr$ with values in $\gr$. Here a  $\gr$-valued 1-cocycle on  $\hr$ with respect to  an indexed open covering $\{\Omega_i\}_{i\in I}$ of $\ho$ is  a collection of smooth maps 
 $$f_{ij}:\hr_{\Omega_j}^{\Omega_i} \to\gr,$$
 satisfying the following cocycle condition:
$\forall \gamma \in \hr_{ij}$ and $\forall \gamma'\in \hr_{jk}$ with $s(\gamma)=r(\gamma')$, we have
\begin{center}
$f_{ij}(\gamma)^{-1}=f_{ji}(\gamma^{-1})$ and $f_{ij}(\gamma)\cdot f_{jk}(\gamma')=f_{ik}(\gamma\cdot \gamma').$
\end{center}
We will denote this data by $f=\{(\Omega_i,f_{ij})\}_{i\in I}$. 

\item A  strict morphism of groupoids 
 
\begin{equation}\nonumber
\xymatrix{
\hr_{\Omega}=\bigsqcup_{i,j}\hr_{\Omega_j}^{\Omega_i} \ar@<.5ex>[d]\ar@<-.5ex>[d]\ar[rr]^-f &&\gr\ar@<.5ex>[d]\ar@<-.5ex>[d]\\
 \bigsqcup_{i}\Omega_{i}\ar[rr]&&\go.
}
\end{equation}
for an open cover  $\Omega= \{\Omega_i\}$ of $\ho$.
 \end{enumerate}

 Associated to  a $\gr$-valued 1-cocycle on  $\hr$, there is a canonical defined  principal $\gr$-bundle over $\hr$.  In fact, any principal $\gr$-bundle over $\hr$ is locally trivial (Cf. \cite{MM}).  

\vspace{2mm}
  
\begin{example}  \begin{enumerate}
\item  (Strict morphisms)
Consider a (strict) morphism of groupoids
\[
\xymatrix{
\hr \ar@<.5ex>[d]\ar@<-.5ex>[d] \ar[r]^f& \gr \ar@<.5ex>[d]\ar@<-.5ex>[d]\\
\ho\ar[r]_{f_0}&\go
}
\]
Using the equivalent definitions 2. or 3. above, it is obviously a generalized morphism by taking $\Omega=\{\ho\}$.
In terms of  the language of principal bundles,  the bi-bundle  is simply given  by $$P_f:=\ho\times_{f_0,t}\gr,$$
with projections $t_f:P_f\to \ho$, projection in the first factor, and 
$s_f:P_f\to \go$, projection using the source map of $\gr$. The actions are the obvious ones, that is,
on the left, $h\cdot (a,g):=(t(h),f(h)\circ g)$ whenever $s(h)=a$ and, on the right, $(a,g)\cdot g':=(a,g\circ g')$ whenever $s(g)=t(g')$.
 \item (Classic principal bundles)
Let $X$ be a manifold and $G$ be a Lie group. By definition a generalized morphism between the unit groupoid $X\rightrightarrows X$ (that is a manifold seen as a Lie groupoid all structural maps are the identity) and the Lie group $G\rightrightarrows \{e\}$ seen as a Lie groupoid is given by a $G$-principal bundle over $X$. 

\end{enumerate}
\end{example}

As the name suggests,  generalized morphism  generalizes the notion of strict morphisms and can be composed. Indeed, if $P$ and $P'$ are generalized morphisms from $\hr$ to $\gr$ and from $\gr$ to $\lr$ respectively, then 
$$P\times_{\gr}P':=P\times_{\go}P'/(p,p')\sim (p\cdot \gamma, \gamma^{-1}\cdot p')$$
is a generalized morphism from $\hr$ to $\lr$.  Consider the category $Grpd_{HS}$ with objects Lie groupoids and morphisms given by isomorphism classes of generalized morphisms. There is a functor
\begin{equation}\label{grpdhs}
Grpd \longrightarrow Grpd_{HS}
\end{equation}
where $Grpd$ is the strict category of groupoids. 

\begin{definition}[Morita equivalent groupoids]
Two groupoids are called Morita equivalent if they are isomorphic in $Grpd_{HS}$. 
\end{definition}

\vspace{2mm}

We list here a few examples of Morita equivalence groupoids which will be used in this paper. 

\begin{example}[Pullback groupoid]
Let $\gr\rightrightarrows \go$ be a Lie groupoid and let $\phi:M\to \go$ be a map such that $t\circ pr_2:M\times_{\go}\gr\to \go$ is a submersion (for instance if $\phi$ is a submersion), then the pullback groupoid $\phi^*\gr:=M\times_{\go}\gr\times_{\go} M\rightrightarrows M$ is Morita equivalent to $\gr$, the strict morphism $\phi^*\gr\to \gr$ being a generalized isomorphism. For more details on this example the reader can see \cite{MM} examples 5.10(4).
\end{example}

\begin{example}[The basic example: the \v{C}ech groupoid]
Given a Lie groupoid $\hr\rightrightarrows\ho$ and an open covering $\{\Omega_i\}_i$ of $\ho$, the canonical strict morphism of groupoids $\hr_{\Omega}\longrightarrow \hr$ is a Morita equivalence. It corresponds to the pullback groupoid by the canonical submersion $\sqcup_i\Omega_i\to \ho$.
\end{example}

\begin{example}[Foliations $\sim$ \'etale groupoids]
 In this paper, one main example to have in mind will be the holonomy groupoid associated to a regular foliation.  Let $M$ be a manifold of dimension $n$. Let $F$ be a subvector bundle of the tangent bundle $TM$.
We say that $F$ is integrable if  
$\ci(F):=\{ X\in \ci(M,TM): \forall x\in M, X_x\in F_x\}$ is a Lie subalgebra of $\ci(M,TM)$. This induces a partition  of $M$ in embedded submanifolds (the leaves of the foliation), given by the solution of integrating $F$. 

The holonomy groupoid of $(M,F)$ is a Lie groupoid  
$$\gr_{M } \rightrightarrows M$$ with Lie algebroid $A\gr=F$ and minimal in the following sense: 
any Lie groupoid integrating the foliation, that is having $F$ as Lie algebroid,  contains an open subgroupoid which maps onto the holonomy groupoid by a smooth morphism of Lie groupoids. 
The holonomy groupoid was constructed by Ehresmann \cite{Ehr} and Winkelnkemper \cite{Win} (see also  \cite{Candel}, \cite{God}, \cite{Pat}).
\end{example}
 
 \subsection{Twistings on  Lie groupoids}
 
 In this paper,  we are only going to consider $PU(H)$-twistings on Lie groupoids 
where $H$ is an infinite dimensional, complex and separable
Hilbert space, and $PU(H)$ is the projective unitary group $PU(H)$  with the topology induced by the
norm topology on the unitary group  $U(H)$. 

\begin{definition}\label{twistedgroupoid}
A  twisting $\alpha$  on a   Lie  groupoid $\gr \rightrightarrows \go$  is given by  a generalized morphism 
\[ \xymatrix{
\alpha: \gr \ar@{-->}[r]  & PU(H).}
\]
Here $PU(H)$ is viewed  as a Lie groupoid with the unit space $\{e\}$. Two twistings 
$\alpha$ and $\alpha'$ are called equivalent if they are  equivalent as generalized morphisms.
\end{definition}
 
 So a twisting on a Lie groupoid $\gr$ is  a locally trivial  right  principal $PU(H)$-bundle $P_{\alpha}$ over $\gr$
hence,  is given by
 a $PU(H)$-valued 1-cocycle on $\gr$
\begin{equation}\label{galphaOmega}
g_{ij}:   \gr_{\Omega_j}^{\Omega^i} \longrightarrow PU(H)
\end{equation}
for an open cover $\Omega= \{\Omega_i\}$ of $\go$. That is,  a twisting  datum $\alpha$ on a Lie  groupoid $ \gr $ is given by a strict morphism of groupoids 
\begin{equation}\label{galpha}
\xymatrix{
\gr_{\Omega}=\bigsqcup_{i,j}\gr_{\Omega_j}^{\Omega_i} \ar@<.5ex>[d]\ar@<-.5ex>[d]\ar[rr]^-f &&PU(H) \ar@<.5ex>[d]\ar@<-.5ex>[d]\\
 \bigsqcup_{i}\Omega_{i}\ar[rr]&&\{e\}.
}
\end{equation}
for an open cover  $\Omega= \{\Omega_i\}$ of $\go$.

 \begin{remark}
 The definition of generalized morphisms given in the last subsection was for two Lie groupoids. The group $PU(H)$ it is not precisely a Lie group but it makes perfectly sense to speak of generalized morphisms from Lie groupoids to this infinite dimensional   groupoid following exactly the same definition,  see (\ref{galphaOmega}) and (\ref{galpha}).
 \end{remark}

\begin{example} \label{example} For a list of  various twistings on some   standard groupoids see example 1.8 in \cite{CaWangAdv}. Here we will only  a few  basic examples used in this paper.
  
 \begin{enumerate}
\item (Twisting on manifolds)  Let $X$ be a $\ci$-manifold. We can consider the  Lie groupoid 
 $X\rightrightarrows X$  where every morphism is the identity over $X$.  A twisting on $X$ is
given by a locally trivial principal $PU(H)$-bundle over $X$, or equivalently,  
 a twisting on $X$ is defined by a strict homomorphism 
\[
\xymatrix{
X_{\Omega}=\bigsqcup_{i,j} \Omega_{i, j} \ar@<.5ex>[d]\ar@<-.5ex>[d]\ar[rr]^-f &&PU(H) \ar@<.5ex>[d]\ar@<-.5ex>[d]\\
 \bigsqcup_{i}\Omega_{i}\ar[rr]&&\{e\}.
}
 \]
with respect to an open cover $\{\Omega_i\}$ of $X$, where $\Omega_{ij} =\Omega_i \cap\Omega_j$. 
Therefore, the restriction of a twisting $\alpha$ on a  Lie groupoid $\gr \rightrightarrows \go$
to its unit $\go$ defines a twisting  $\alpha_0$ on the manifold $\go$.

\item\label{obundle} (Orientation twisting) Let $X$ be a  manifold with an oriented real vector bundle $E$. The  bundle $E \to X$ defines
a natural generalized morphism 
\[
\xymatrix{
X\ar@{-->}[r] & SO(n).}
\]
Note that the fundamental spinor  representation of   $Spin^c(n)$ gives rise to a commutative
diagram of Lie group homomorphisms
\[
\xymatrix{
Spin^c(n) \ar[d]   \ar[r] & U(\CC^{2^{[n/2]}}) \ar[d] \\
SO(n) \ar[r] & PU(\CC^{2^{[n/2]}}).}
\]
With a choice of inclusion $\CC^{2^{[n/2]}}$ into a Hilbert space $H$, we have a canonical
twisting, called the orientation twisting, denoted by
\begin{equation}\label{otwistingX}
\xymatrix{
\o_{E}:  X\ar@{-->}[r] & PU(H).}
\end{equation}
If now $\gr\rightrightarrows X$ is a Lie groupoid and $E$ is an oriented $\gr$-vector bundle over $X$, we have in the same way an orientation twisting
\begin{equation}\label{otwisting}
\xymatrix{
\o_{E}:  \gr\ar@{-->}[r] &SO(n)\ar[r]& PU(H)}
\end{equation}
in the case where $E$ admits an $\gr$-invariant metric. In particular when $\gr$ acts properly on $P$ and on $E$, \cite{PPT} proposition 3.14 and \cite{HF} theorem 4.3.4. 


\item (Pull-back twisting) Given a twisting $\alpha$ on $ \gr$ and  for any generalized 
homomorphism $\phi: \hr \to \gr$, there is a pull-back twisting 
\[\xymatrix{
\phi^*\alpha:  \hr  \ar@{-->}[r]  & PU(H)}
\]
defined by the composition of $\phi$ and $\alpha$.  In particular, 
for a continuous map $\phi: X\to Y$, a twisting $\alpha$ on $Y$ gives a pull-back twisting 
$\phi^*\alpha$ on $X$. The principal $PU(H)$-bundle over $X$ defines by $\phi^*\alpha$ is
the pull-back of the  principal $PU(H)$-bundle on $Y$ associated to $\alpha$.

\item (Twisting on fiber product groupoid)  Let $N\stackrel{p}{\rightarrow} M$ be a submersion. We consider the fiber product $N\times_M N:=\{ (n,n')\in N\times N :p(n)=p(n') \}$,which is a manifold because $p$ is a submersion. We can then take the groupoid 
$$N\times_M N\rightrightarrows N$$ which is  a subgroupoid of the pair groupoid 
$N\times N \rightrightarrows N$.  Note that this groupoid is in fact Morita equivalent to the groupoid $ M \rightrightarrows M$.  A twisting on  
 $N\times_M N\rightrightarrows N$ is  given by
a pull-back twisting from a   twisting on  $M$.  

\item (Twisting on the space of leaves of a foliation)
Let $(M,F)$ be a regular foliation with holonomy groupoid $\gr_{M }$. A twisting on the space of leaves  is by definition a twisting on the holonomy groupoid $\gr_{M }$. We will often use the notation 
\[\xymatrix{
M/F \ar@{-->}[r] & PU(H)}
\]
for the corresponding  generalized morphism.

Notice that by definition a twisting on the spaces of leaves is a twisting on the base $M$ which admits a compatible action of the holonomy groupoid. It is however not enough to have a twisting on base which is leafwisely constant, see for instance remark 1.4 (c) in \cite{HS}.

\end{enumerate}
\end{example}

   A twisting on a Lie groupoid $\gr \rightrightarrows M$ gives rise to an
  $U(1)$-central extension over the Morita equivalent groupoid $\gr_{\Omega}$ by
  pull-back the $U(1)$-central extension of $PU(H)$
  \[
  1\to U(1) \to U(H)  \to PU(H) \to 1.
  \]
  We will not call an $U(1)$-central extension of a Morita equivalent groupoid of 
  $ \gr$ a twisting on $\gr$ as in \cite{TXL}.  This is due to the fact that the associated principal  $PU(H)$-bundle might depend on the choice of Morita equivalence bibundles,  even though
  the isomorphism class of principal  $PU(H)$-bundle does not  depend on the choice of Morita equivalence bibundles.  It is important in applications to remember the  $PU(H)$-bundle 
  arising from a twisting, not just its isomorphism class.

Denote by $Tw(\gr)$ the set of  equivalence classes of  twistings on $\gr$.   There is a canonical abelian group structure on $Tw(\gr)$ as follows.    Fix an isomorphism $H\otimes H \to H$,  we have a group homomorphism
 \[
 m: PU(H) \times PU(H) \longrightarrow PU(H\otimes H)  \cong PU(H). 
 \]
 Then given two twistings $\alpha$ and $\beta$ on $\gr $, we can define 
 \begin{equation}\label{sumcocycles}
 \xymatrix{
 \alpha +\beta :  \gr   \ar@{-->}[r]^{(\alpha, \beta)\qquad } &  PU(H)   \times PU(H)  \ar[r]^{\qquad m}  & PU(H).}
\end{equation}
In terms of  $U(1)$-central extension   over the Morita equivalent groupoid $\gr_{\Omega}$, we can choose a common open 
cover $\Omega$ of $\go$ such that $\alpha$ and $\beta$ define $U(1)$-central extensions
\[
S^1 \longrightarrow   R_\beta  \longrightarrow  \gr_{\Omega} \qquad \text{and} \qquad 
S^1 \longrightarrow   R_\beta \longrightarrow  \gr_{\Omega} 
\]
respectively.   Then $\alpha + \beta$  corresponds to the tensor product of the two extensions.  See \cite{TXL} for more discussions of
twistings using the language of $U(1)$-central extensions.  

\begin{remark}  Let $P_\alpha$ be the principal $PU(H)$-bundle over $\gr$ defined a twisting $\alpha$, and $K(H)$ be the elementary $C^*$-algebra of the compact operators on $H$. There is an associated bundle of  elementary $C^*$-algebras over $\gr$
defined by
\[
A_\alpha = P_\alpha\times_{Ad} K(H) \longrightarrow \go
\]
where $Ad$ denotes the adjoint action of $PU(H)$ on $K(H)$. The bundle  $ A_\alpha\to \go$ 
satisfies  FellÕs condition and continuous actions  of  $\gr$ in the sense of \cite{KMRW}, where the Brauer group  $Br(\gr)$ of $\gr$ is defined to be the group of Morita equivalence classes of  elementary $C^*$-algebras over $\gr$.    Then the addition structure on $Tw(\gr)$ corresponds to the tensor product of bundles of  elementary $C^*$-algebras over $\gr$.   Therefore, there is a canonical isomorphism between  $Tw(\gr)$ and the Brauer group $Br(\gr)$. 
\end{remark}

\section{Twisted deformation indices}

\subsection{Twisted groupoid's $C^*$-algebras}
Let $(\gr,\alpha)$ be a twisted groupoid. With respect to a covering  $\Omega = \{\Omega_i\}$ of $\go$, the twisting $\alpha$ is given by a strict morphism of groupoids 
$$ \alpha: \gr_{\Omega}   \longrightarrow PU(H),$$
where $\gr_{\Omega}$ is the covering groupoid associated to $\Omega$. 
Consider the central extension of groups
$$S^1 \longrightarrow U(H) \longrightarrow PU(H),$$
 we can pull it back to get a $S^1$-central extension of Lie groupoid $R_{\alpha}$  over $\gr_{\Omega}$ 
\begin{equation}
\xymatrix{
  S^1\ar[d]\ar[r]& S^1\ar[d]\\
R_{\alpha}\ar[d]\ar[r]&U(H)\ar[d]\\
\gr_{\Omega}\ar[r]_-{\alpha}&PU(H)\\
}
\end{equation}
In particular, $R_{\alpha}\rightrightarrows \bigsqcup_i\Omega_i$ is a Lie groupoid and $R_{\alpha}\longrightarrow \gr_{\Omega}$ is a $S^1$-principal bundle.

We recall the definition of the convolution  algebra and the $C^*$-algebra of a twisted Lie groupoid $(\gr, \a)$ \cite{Ren87,TXL}:

\begin{definition}
Let $R_{\alpha}$ be the $S^1$-central extension of groupoids associated to a twisting $\alpha$. The convolution algebra of $(\gr, \a)$ is by definition the following sub-algebra of 
$C_{c}^{\infty}(R_{\alpha})$: 
\begin{equation}
C_{c}^{\infty}(\gr, \a)=\{f\in C_{c}^{\infty}(R_{\alpha}): f(\tilde{\gamma} \cdot \lambda)=\lambda^{-1} f(\tilde{\gamma}), \forall \tilde{\gamma}\in R_{\alpha },\forall \lambda \in S^1\}.
\end{equation}
The maximal(reduced resp.) $C^*$-algebra of $(\gr, \alpha)$, denoted by $C^*(\gr,  \alpha)$ ($C^*_r(\gr, \alpha)$ resp.),  is the completion of 
$C_{c}^{\infty}(\gr, \a)$ in $C^*(R_{\alpha})$ ($C_r^*(R_{\alpha})$ resp.).
\end{definition}

Let $L_{\alpha}:=R_{\alpha}\times_{S^1}\mathbb{C}$ be the  complex line bundle  over $\gr_\Omega$ which can be considered as a Fell bundle (using the groupoid structure of $R_{\alpha}$) over $\gr_{\Omega}$. 
Then the algebra of  compactly supported smooth sections of this Fell bundle, denoted by  $C_{c}^{\infty}(\gr_{\Omega },L_{\alpha})$, is isomorphic to $C_{c}^{\infty}(\gr, \a)$. Therefore as  $C^*$-algebras, 
\[
C^*(\gr_{\Omega}, L_{\alpha})\cong C^*(\gr, \a), 
\]
 see (23) in \cite{TXL} for an explicit isomorphism.

\begin{remark}\label{decomposition} (\cite{TXL})  
Given the extension $R_{\alpha}$ as above, the $S^1$-action on $R_{\alpha}$ induces a $\mathbb{Z}$-grading in $C^*(R_{\alpha})$ (Proposition 3.2, ref.cit.). More precisely,  we have
\begin{equation}\label{gradext}
C^*(R_{\alpha})\cong \bigoplus_{n\in \mathbb{Z}} C^*(\gr, n  \alpha )
\end{equation}
where $C^*(\gr, n \alpha )$ is the maximal $C^*$-algebra of the twisted groupoid $(\gr,  n \a )$ 
corresponding to the  Fell bundle
$$ L_{\alpha}^n=L_{\alpha}^{\otimes n} \longrightarrow \gr_{\Omega },$$
for all $n\neq 0$, and $C^*(\gr, \alpha^0)=C^*(\gr_{\Omega})$ by convention. Similar results hold for the reduced $C^*$-algebra. 
\end{remark}

\begin{definition}\label{twistedkth} 
Following \cite{TXL}, we define the twisted K-theory of the twisted groupoid $(\gr,\alpha)$ by
\begin{equation}
K^i(\gr,\alpha):=K_{-i}(C^*(\gr, \a)).
\end{equation}
In particular if $\alpha$ is trivial we will be using the notation (unless specified otherwise) $K^i(\gr)$ for the respective $K$-theory group of the maximal groupoid $C^*$-algebra.
\end{definition}

By the next lemma, the group $K^i(\gr,\alpha)$ is well defined, up to a canonical isomorphism coming from the respective Morita equivalences.

\begin{lemma}
Let $\gr$ be a Lie groupoid. Let $\alpha_1,\alpha_2:\gr---> PU(H)$ be two twistings on $\gr$. Suppose we have a given isomorphism $\eta: P_{\alpha_1}\cong P_{\alpha_2}$ between the principal bundles associated to $\alpha_1$ and $\alpha_2$. We have an induced isomorphism between the respective twisted K-theory groups:
\begin{equation}
\xymatrix{
K^*(\gr,\alpha_1)\ar[r]^-{\eta_*}_-{\cong}&K^*(\gr,\alpha_2)
}
\end{equation}
\end{lemma}

\begin{proof}
The fact that the K-theory groups are isomorphic follows from \cite{Ren87} or \cite{Muh99} theorem 11 (or proposition 3.3 in \cite{TXL}). We want here to emphasize how $\eta$ induces such an explicit isomorphism. Indeed, the isomorphism $\eta$ between $P_{\alpha_1}$ and $P_{\alpha_2}$ is equivalent to an equivalence between cocycles $\gr_{\Omega_1}\stackrel{\alpha_1}{\longrightarrow} PU(H)$ and $\gr_{\Omega_2}\stackrel{\alpha_2}{\longrightarrow} PU(H)$ representing respectively such principal bundles. Thus giving 
$\eta$ is equivalent to give a common refinement $\Omega$ of $\Omega_1$ and $\Omega_2$ together with a common cocycle extension, {\it i.e.}, a cocycle $\gr_\Omega\stackrel{\alpha}{\longrightarrow} PU(H)$ with $\alpha|_{\Omega_i}=\alpha_i$, $i=1,2$. Then, by taking the respective $S^1$-central extensions, we have Morita equivalences of extensions
$$
\xymatrix{
R_{\alpha_1}\ar[r]^-\sim &R_\alpha &R_{\alpha_2}\ar[l]_-\sim
}
$$
induced by pullback from the Morita equivalences
$$
\xymatrix{
\gr_{\Omega_1}\ar[r]^-\sim &\gr_\Omega &\gr_{\Omega_2}\ar[l]_-\sim.
}
$$
Hence, $\eta$ induce an explicit Morita equivalence of $S^1$-central extensions between $R_{\alpha_1}$ and $R_{\alpha_2}$ giving then an explicit isomorphism between the respective $K-$theory groups.

\end{proof}

We will also need to understand the compatibility of twisted 
$K$-theory with Morita equivalence, more explicitly:

\begin{lemma}\label{etaMorita}
Let $\gr$ and $\gr'$ be two Morita equivalent groupoids. Let us denote by $\gr\stackrel{\phi}{--->}\gr'$ the generalized isomorphism. Consider  two twisting $\alpha'_1,\alpha'_2:\gr'---> PU(H)$ on $\gr'$ and denote by $\alpha_i:=\alpha'_i\circ \phi$ the induced twistings on $\gr$. Suppose we have a given isomorphism $\eta: P_{\alpha_1}\cong P_{\alpha_2}$ between the principal bundles associated to $\alpha_1$ and $\alpha_2$. We have the following commutative diagram of K-theory group isomorphisms:
\begin{equation}
\xymatrix{
K^*(\gr,\alpha_1)\ar[d]_-{\eta_*}^-{\cong}\ar[r]^-{\phi_*}_-{\cong}&K^*(\gr',\alpha_1')\ar[d]^-{\phi(\eta)_*}_-{\cong}\\
K^*(\gr,\alpha_2)\ar[r]_-{\phi_*}^-{\cong}&K^*(\gr',\alpha_2')
}
\end{equation}
\end{lemma}

\begin{proof}
The generalized isomorphism $\phi$ induces a generalized isomorphism
\[
\gr_\Omega \stackrel{\overline{\phi}}{--->} \gr_\Omega'
\]
as a composition of  generalized isomorphisms  for any given open covers $\Omega$ and 
$\Omega'$. Now, if we consider two cocycles $\gr_{\Omega}\stackrel{\alpha}{\longrightarrow} PU(H)$ and $\gr'_{\Omega'}\stackrel{\alpha'}{\longrightarrow} PU(H)$ representing two principal bundles $P_\alpha$ and $P_{\alpha'}$ with 
$P_\phi\times_{\gr'}P_\alpha'\cong P_\alpha$, we have by definition that $\alpha'\circ \overline{\phi}=\alpha$ and thus we have an induced generalized isomorphism of extensions between the respective pullback extensions
$$R_\alpha\stackrel{\tilde{\phi}}{--->}R_{\alpha'}.$$

Coming back to the notations of the lemma, we will  denote by $\alpha$ the common cocycle extension of $\alpha_1$ and $\alpha_2$ induced by $\eta$ and by $\alpha'$ 
the cocycle such that $\alpha'\circ \phi=\alpha$, then it is by definition the common cocycle extension of $\alpha_1'$ and $\alpha_2'$ induced by $\phi(\eta)$ (which is by definition the isomorphism $Id\times_\gr \eta$ between $P_{\alpha_1'}=P_{\phi^{-1}}\times_\gr P_{\alpha_1}$ and $P_{\alpha_2'}=P_{\phi^{-1}}\times_\gr P_{\alpha_2}$).
We have the following commutative diagram of extension's generalized isomorphisms
\begin{equation}\nonumber
\xymatrix{
R_{\alpha_1}\ar@/^2pc/[rr]^-{\eta}_-{\sim} \ar[d]_-{\tilde{\phi_1}}^-{\sim}\ar[r]^-\sim &R_\alpha\ar[d]_-{\tilde{\phi}}^-{\sim} &R_{\alpha_2}\ar[d]^-{\tilde{\phi_2}}_-{\sim}\ar[l]_-\sim
\\
R_{\alpha_1'}\ar@/_2pc/[rr]_-{\phi(\eta)}^-{\sim}\ar[r]_-\sim &R_{\alpha'} &R_{\alpha_2'}\ar[l]^-\sim
}
\end{equation}
which implies the desired result.

\end{proof}


\begin{remark}
For  the groupoid  given by  a manifold $M\rightrightarrows M$. A twisting on $M$ can be given by a Dixmier-Douday class on $H^3(M,\ZZ)$. In this event, the twisted K-theory, as we defined it, coincides with  twisted K-theory defined in \cite{ASeg,Kar08}. Indeed the $C^*$-algebra $C^*(M,\alpha)$ is Morita equivalent to the continuous trace $C^*$-algebra defined by the corresponding Dixmier-Douady class (see for instance Theorem 1 in \cite{FMW}).
\end{remark}

\subsection{Index morphism associated to an immersion of groupoids}
We  briefly discuss here the deformation groupoid of an immersion of groupoids which is called  the normal groupoid in   \cite{HS}.

Given  an immersion of Lie groupoids $\gr_1\stackrel{\varphi}{\rightarrow}\gr_2$, let 
$\gr^N_1=\mathscr{N}_{\gr_1}^{\gr_2}$ be the total space of the normal bundle to $\varphi$, and $(\gr_{1}^{(0)})^N$ be the total space of the normal bundle to $\varphi_0: \go_1 \to \go_2$.  Consider $\gr^N_1\rightrightarrows (\gr_{1}^{(0)})^N$ with the following structure maps: The source map is the derivation in the normal direction 
$d_Ns:\gr^N_1\rightarrow (\gr_{1}^{(0)})^N$ of the source map (seen as a pair of maps) $s:(\gr_2,\gr_1)\rightarrow (\gr_{2}^{(0)},\gr_{1}^{(0)})$ and similarly for the target map.

 As remarked by Hilsum-Skandalis (remarks 3.1, 3.19 in \cite{HS}), $\gr^N_1$ may fail to inherit a Lie groupoid structure (see counterexample just before section IV in \cite{HS}). A sufficient condition is when $(\gr_{1}^{(0)})^N$ is a $\gr_1$-vector bundle over $\gr_{1}^{(0)}$. This is the case when $\gr_1^x\to \gr_{2}^{\varphi(x)}$ is \'etale for every $x\in \gr_{1}^{(0)}$ (in particular if the groupoids are \'etale) or when one considers a manifold with two foliations $F_1\subset F_2$ and the induced immersion  (again 3.1, 3.19 in \cite{HS}).

 The deformation to the normal bundle construction allows us to consider a $C^{\infty}$ structure on 
$$
\gr_{\varphi}:=\left( \gr^N_1\times \{0\} \right) \bigsqcup  \left( \gr_2\times \mathbb{R}^*\right),
$$
such that $\gr^N_1\times \{0\}$ is a closed saturated submanifold and so $\gr_2\times \mathbb{R}^*$ is an open submanifold.
The following results are  an immediate consequence of the functoriality of the deformation to the normal cone construction.

\begin{proposition}[Hilsum-Skandalis, 3.1, 3.19 \cite{HS}]\label{HSimmer}
Consider an immersion $\gr_1\stackrel{\varphi}{\rightarrow}\gr_2$ as above for which $(\gr_1)^N$ inherits a Lie groupoid structure.
Let $\gr_{\varphi_0}:= \big( (\gr_{1}^{(0)})^N\times \{0\} \big) \bigsqcup \big(  \gr_{2}^{(0)}\times \mathbb{R}^* \big)$ be the deformation to the normal cone of  the  pair $(\gr_{2}^{(0)},\gr_{1}^{(0)})$. The groupoid
\begin{equation}
\gr_{\varphi}\rightrightarrows\gr_{\varphi_0}
\end{equation}
with structure maps compatible  with the ones of the groupoids $\gr_2\rightrightarrows \gr_{2}^{(0)}$ and $\gr_1^N\rightrightarrows (\gr_{1}^{(0)})^N$, is a Lie groupoid with $C^{\infty}$-structures coming from  the deformation to the normal cone.
\end{proposition}

One of the interest of these kind of groupoids is to be able to define deformation indices.
Indeed, restricting the deformation to the normal cone construction to the closed interval $[0,1]$ and since the groupoid $\gr_2 \times (0,1]$ is an open saturated subgroupoid of $\gr_{\varphi}$ (see 2.4 in \cite{HS} or \cite{Ren} for more details), we have a short exact sequence of $C^*-$algebras
\begin{equation}\label{segtimm}
0 \rightarrow C^*(\gr_2 \times (0,1]) \longrightarrow C^*(\gr_{\varphi})
\stackrel{ev_0}{\longrightarrow}
C^*(\gr_1^N) \rightarrow 0,
\end{equation}
with $C^*(\gr_2 \times (0,1])$ contractible. Then the 6-term exact sequence in K-theory  provide   the isomorphism
\[
(ev_0)_{*}:  K_*(C^*(\gr_\varphi))   \cong K_*(C^*(\gr_1^N)). 
\]
 Hence
we can define the index morphism
$$D_{\varphi}:K_*(C^*(\gr_1^N))\longrightarrow K_*(C^*(\gr_2))$$
between the K-theories of the maximal $C^*$-algebras as the induced deformation index morphism 
\[
D_{\varphi}:=(ev_1)_*\circ(ev_0)_{*}^{-1}: K_*(C^*(\gr_1^N)) \cong K_*(C^*(\gr_\varphi))  \longrightarrow  K_*(C^*(\gr_2)) .
\]

\subsection{The index of a groupoid immersion with a twisting}\label{propimm}
Now Consider an immersion of Lie groupoids $\gr_1\stackrel{\varphi}{\rightarrow}\gr_2$  with  a twisting $\alpha$ on $\gr_2$  for which $(\gr_1)^N$ inherits a Lie groupoid structure. We will see that we can still define index morphisms. First we prove the following elementary  result.

\begin{proposition}\label{deftwistimm}
Given an immersion of Lie groupoids $\gr_1\stackrel{\varphi}{\rightarrow}\gr_2$  as above and a twisting $\alpha$ on $\gr_2$.  There is a canonical twisting $\alpha_\varphi$ on the Lie groupoid 
$\gr_{\varphi} \rightrightarrows \gr_{\varphi_0}$,  extending the pull-back  twisting on $\gr_2\times \mathbb{R}^*$ from $\alpha$.
\end{proposition}

\begin{proof}
The proof is a simple application of the functoriality of the deformation to the normal cone construction. Indeed, the twisting $\alpha$ on $\gr_2$ induces by pullback (or composition of cocycles) a twisting $\alpha\circ \varphi$ on $\gr_1$. The twisting $\alpha$ on $\gr_2$ is given by a $PU(H)$-principal bundle $P_{\alpha}$ with a compatible left action of $\gr_2$, and by definition the twisting $\alpha\circ \varphi$ on $\gr_1$ is given by the pullback of $P_{\alpha}$ by 
$\varphi_0: \gr_{1}^{(0)} \to \gr_{2}^{(0)}$. In particular,  we have $P_{\alpha\circ \varphi}=\gr_{1}^{(0)}\times_{\gr_{2}^{(0)}}P_{\alpha}$. 
Hence the action map 
$\gr_2\times_{\gr_{2}^{(0)}}P_{\alpha}\to P_{\alpha}$ can be considered as   in the category of pairs:

$$(\gr_2\times_{\gr_{2}^{(0)}}P_{\alpha},\gr_1\times_{\gr_{1}^{(0)}}P_{\alpha\circ\varphi})\longrightarrow
(\gr_{2}^{(0)}\times_{\gr_{2}^{(0)}}P_{\alpha},\gr_{1}^{(0)}\times_{\gr_{1}^{(0)}}P_{\alpha\circ \varphi}).$$

We can then apply the deformation to the normal cone functor to obtain the desire $PU(H)$-principal bundle with a compatible $\gr_{\varphi}$-action, which gives the desired twisting.
\end{proof}

We will now define the index morphism associated to an immersion 
$\gr_1\rightarrow \gr_2$ as above under the presence of a twisting on $\gr_2$. Associated to the twisted groupoid $(\gr_\varphi,\alpha_\varphi)$ of the last proposition there is an $S^1$-central extension 
$R_{\alpha_\varphi}$ which has an open dense subextension $R_{\alpha_{(0,1]}}$, the $S^1$-central extension associated to $(\gr_2\times (0,1], \alpha_{(0,1]})$ where $\alpha_{(0,1]}$ the twisting giving by the projection $\gr_2 \times (0,1]\to\gr_2$. Denoting $\alpha^N:=\alpha_{\varphi}|_{\gr_1^N}$, there is a short exact sequence of $C^*-$algebras
\begin{equation}\label{twistedses}
0 \rightarrow C^*(R_{\alpha_{(0,1]}}) \longrightarrow C^*(R_{\alpha_{\varphi}})
\stackrel{ev_0}{\longrightarrow}
C^*(R_{\alpha^N}) \rightarrow 0,
\end{equation}
which respects the $\mathbb{Z}$-grading (\ref{gradext}) and it defines thus a short exact sequence of $C^*-$algebras
\begin{equation}\label{twistedses1}
0 \rightarrow C^*(\gr_2 \times (0,1],\alpha_{(0,1]}) \longrightarrow C^*(\gr_{\varphi},\alpha_{\varphi})
\stackrel{ev_0}{\longrightarrow}
C^*(\gr_1^N,\alpha^N) \rightarrow 0.
\end{equation}
The disintegration results in \cite{Ren87} also conclude the same result directly with the Fell bundle's algebras without passing through the extensions. Hence
we can define the index morphism 
$$D_{\varphi}:K_*(C^*(\gr_1^N,\alpha^N))\longrightarrow K_*(C^*(\gr_2,\alpha))$$
between the K-theories of the maximal $C^*$-algebras as the induced deformation morphism $Ind_{\varphi}:=(ev_1)_*\circ(ev_0)_{*}^{-1}$ exactly as in the untwisted case. 

\section{Groupoid equivariant pushforward and wrong way functoriality}\label{pushsection}
Let  $\gr\rightrightarrows M$  be a Lie groupoid with a given twisting $\alpha$.
A  $\gr$-manifold  $P$ is a smooth manifold $P$  with  a momentum map $\pi_P:P\to M$, which is assumed to be an oriented submersion,  and  a right action of $\gr$ on $P$:  $P\rtimes \gr  \to P $  given by
 $(p, \gamma) = p\circ \gamma$ such that
 \[
 (p\circ \gamma_1) \circ \gamma_2 = p\circ (\gamma_1 \cdot \gamma_2)
 \]
 for any $(\gamma_1, \gamma_2) \in \gr^{(2)}$.
  Here $P\rtimes \gr =\{ (p, \gamma) \in P\times \gr |  \pi_P (p) = r(\gamma)\}.$ 
 We will denote by $T^vP$ the vertical tangent bundle associated to $\pi_P$. 
 A  $\gr$-manifold  $P$ is called $\gr$-proper if the map
 \[
 P\rtimes \gr \longrightarrow P\times P
 \]
 defined by $(p, \gamma) \mapsto (p,  p\circ \gamma)$ is  proper. Then  the induced action groupoid
 \[
  P\rtimes \gr \rightrightarrows  P
  \]
  with $s(p, \gamma) =p, r(p, \gamma)=  p\circ \gamma $ is a proper Lie groupoid. 
 
{\bf Hypothesis:} In what follows, for any $\gr$-manifold $P$ as above, we will assume that $T^vP$ is oriented and that it admits a $\gr$-invariant metric. This is the case when $\gr$ acts on 
$P$ properly.
We will  construct  the twisted geometric K-homology group and the Baum-Connes assembly map
under this assumption.

Let $P,N$ be two proper  $\gr$-manifolds and $f: P\longrightarrow N$ be  a smooth oriented $\gr$-equivariant map with a twisting $\alpha$ on $N\rtimes \gr$. 
Using only geometric deformation groupoids, we will construct a morphism, called the shriek map $f_!$
\begin{equation}\label{fshriekintrosection}
\xymatrix{
K^*(P\rtimes \gr, \alpha+\o_f)\ar[r]^-{f_!}&K^*(N\rtimes \gr,\alpha)
}
\end{equation}
where $\o_f$ is the twisting over $P\rtimes \gr$ given by the 
$\gr$-vector bundle $f^*T^vN\oplus T^vP$.
The main result of this section is the functoriality of this shriek map. A main ingredient in the construction is  the twisted equivariant Thom isomorphism which is reviewed in the appendix. 

We  recall the definition of the semi-direct product groupoid. Consider a Lie groupoid $H_A\rightrightarrows A$, we say that it is a $\gr$-groupoid if $\gr$ acts on $H_A$ and  $A$ such that  the source and target maps of $H_A$ are $\gr-$equivariant. Under this situation we might form the semi-direct product groupoid $$H_A\rtimes \gr\rightrightarrows A.$$ Typically, but not exclusively, $H_A\rightrightarrows A$ will be a $\gr-$vector bundle $E$ over $P$ considered as groupoid   $E\rightrightarrows P$ or $E\rightrightarrows E$ considered as a manifold. We will mention every time, if not obvious, which case
we are considering.

\subsection{Twisted wrong way functoriality for $\gr$-manifolds}

The construction of the shriek map (\ref{fshriekintrosection}) follows the lines of Connes construction, II.6 in  \cite{Concg}, see also \cite{Shim} for a more complete description in the K-oriented untwisted case. It is divided in four steps.

{\bf Step 1.} The first step is the twisted $\gr$-equivariant Thom isomorphism associated to the vector bundle $T^vP\to P$, applied to the twisting $\alpha+\o_f$ over $P\rtimes \gr$  

\begin{equation}\label{Step1map}
\xymatrix{
K^*(P\rtimes \gr,\alpha+\o_f)\ar[r]^-{\mathscr{T}_{T^vP}^{\gr}}&K^*(T^vP\rtimes \gr,\alpha+\o_{f^*T^vN}).
}
\end{equation} 
Indeed this is due to the fact that $\o_f+\o_{T^vP}$ is canonically homotopic (as twistings) to $\o_{f^*T^vN}$.

{\bf Step 2.} The second step is the twisted equivariant Thom isomorphism associated to the action (as a groupoid) of $T^vP$ on $(f^*T^vN)^*$, that is, the Thom  isomorphism associated to the $T^vP\rtimes \gr$-vector bundle $(f^*T^vN)^*$ over $P$, applied to the twisting 
$\alpha+\o_{f^*T^vN}$ 
\begin{equation}\label{Step2map}
\xymatrix{
K^*(T^vP\rtimes \gr,\alpha+\o_{f^*T^vN})\ar[rr]^-{\mathscr{T}_{(f^*T^vN)^*}^{T^vP\rtimes\gr}}&&
K^*((f^*T^vN)^*\rtimes (T^vP\rtimes \gr),\alpha+ \o_{f^*T^vN}+\o_{(f^*T^vN)^*})
}
\end{equation}

{\bf Step 3.} The third step is the isomorphism in twisted $K$-theory
\begin{equation}\label{Step3map}
\xymatrix{
K^*((f^*T^vN)^*\rtimes (T^vP\rtimes \gr),\alpha+ \o_{f^*T^vN}+\o_{(f^*T^vN)^*})\ar[rr]^-{\mathscr{F}}&&K^*(f^*T^vN\rtimes (T^vP\rtimes \gr),\alpha)
}
\end{equation}
induced by the Fourier isomorphism of $C^*$-algebras, proposition 2.12 \cite{CaWangAdv},

\begin{equation}\label{FourierStep3map}
\xymatrix{
C^*((f^*T^vN)^*\rtimes (T^vP\rtimes \gr),\alpha+ \o_{f^*T^vN}+\o_{(f^*T^vN)^*})\ar[r]^-{F}&C^*(f^*T^vN\rtimes (T^vP\rtimes \gr),\alpha)
}
\end{equation}
where the first groupoid is obtained from the semi-direct product of $T^vP\rtimes \gr\rightrightarrows P$ acting on $(f^*T^vN)^*\rightrightarrows P$ and the second is obtained from the semi-direct product  of $T^vP\rtimes \gr\rightrightarrows P$ acting on $f^*T^vN\rightrightarrows f^*T^vN$.

{\bf Step 4.}
The final step is to consider the groupoid immersion
\begin{equation}
\xymatrix{
P\ar[r]^-{f\times \triangle} & N\times_M(P\times_M P). 
}
\end{equation}
The associated  deformation groupoid is  $\gr_f\rightrightarrows \gr_{f}^{(0)}$ 
where
\begin{equation}
\gr_f:=f^*(T^vN)\rtimes T^vP\times \{0\}\bigsqcup N\times_M(P\times_M P) \times (0,1]\, \text{and}
\end{equation}
\begin{equation}
\gr_{f}^{(0)}=f^*T^vN\times \{0\}\bigsqcup N\times_M P\times (0,1]
\end{equation}
Notice that $N\times_M(P\times_M P)$ and $N$ are Morita equivalent groupoids with the  Morita equivalence   given by the canonical projection.

The functoriality of the deformation to the normal cone construction yields an action of $\gr$ on $\gr_f$.
Let $\alpha_f$ be the twisting on $\gr_f\rtimes \gr$ given by proposition \ref{deftwistimm}. It is immediate to check that $$\alpha_f|_{(f^*(T^vN)\rtimes T^vP)\rtimes \gr}
=\pi_{f^*T^vN\rtimes T^vP}^*\alpha.$$

We can hence consider the twisted deformation index morphism associated to $(\gr_f\rtimes \gr,\alpha_f)$ :

\begin{equation}\label{Step4map}
\xymatrix{
K^*(f^*T^vN\rtimes (T^vP\rtimes \gr),\alpha)\ar[r]^-{D_f}&K^*(N\times_M(P\times_M P)\rtimes\gr,\alpha)\ar[d]^-{m}_-{\cong}\\&K^*(N\rtimes\gr,\alpha)
}
\end{equation} 
where we denoted $D_f$ instead of $D_{f\times \triangle}$ for keeping the notation short, and $m$ is the isomorphism defined by the Morita equivalence between $N\times_M(P\times_M P)\rtimes\gr$ and $N$.

\begin{definition}[Pushforward morphism]
Let $P,N$ be two manifolds and $f:P\longrightarrow N$ be a smooth oriented $\gr$-equivariant map \footnote{Remember we are assuming that both $T^vP$ and $T^vN$ admit a $\gr$-invariant metric}. Under the presence of a twisting $\alpha$ on $ \gr$ we let
\begin{equation}
\xymatrix{
K^*(P\rtimes \gr,  \alpha+\o_f) \ar[r]^-{f_!}&K^*(N\rtimes\gr,\alpha)
}
\end{equation}
 be the morphism given by the composition of the morphisms given in the three last steps, that is, the morphism (\ref{Step1map}) followed by (\ref{Step2map}) followed by (\ref{Step4map}). By definition $f_!$ fits in the fowolling commutative diagram:
\begin{equation}\label{diagpush}
{\small 
\xymatrix{
K^*(P\rtimes \gr,   \alpha+\o_f)\ar[rrdd]_-{f_!}\ar[r]^-{\mathscr{T}}&K^*(T^vP\rtimes \gr, \alpha+\o_{f^*T^vN})\ar[r]^-{\mathscr{T}_F}&
K^*(f^*T^vN\rtimes T^vP\rtimes \gr,\alpha)\ar[d]^-{D_f}\\&&K^*(N\times_M(P\times_M P)\rtimes\gr,\alpha)\ar[d]^-{m}_-{\cong}\\&&K^*(N\rtimes\gr,\alpha)
}}
\end{equation}
\end{definition}

where $\mathscr{T}_F$ will denote the Thom isomorphism from (\ref{Step2map}) followed by the Fourier isomorphism (\ref{Step3map}).

Our first main result is the wrong way functoriality of the precedent construction.

\begin{theorem}\label{twistfun}
The above push-forward morphism is functorial, that means, if we have a composition of smooth $\gr$-maps between $\gr-$manifolds as above:
\begin{equation}
P\stackrel{f}{\longrightarrow}N\stackrel{g}{\longrightarrow}L,
\end{equation}

and a twisting  $\xymatrix{\alpha:\gr \ar@{-->}[r] &PU(H)}$,  then the following diagram commutes
\[
\xymatrix{
K^*(P\rtimes \gr,\alpha+\o_{g\circ f}) \ar[rr]^-{(g\circ f)_!}\ar[rd]_-{f_!}&&K^*(L\rtimes \gr,\alpha)\\
&K^*(N\rtimes \gr,\alpha+\o_g)\ar[ru]_-{g_!}&
}
\]

\end{theorem}

\begin{proof}

Let us recall the notations and definitions we used above to define the shriek maps:
$f_!:=m\circ D_f\circ \mathscr{T}_F\circ \mathscr{T}$, $g_!:=m\circ D_g\circ \mathscr{T}_F\circ \mathscr{T}$ and 
$(g\circ f)!:=m\circ D_{g\circ f}\circ \mathscr{T}_F\circ \mathscr{T}$, where $m$ stand for the Morita isomorphisms (induced by Morita equivalences) and $\mathscr{T}$, $\mathscr{T}_F$ for the Thom isomorphisms respectively ($\mathscr{T}_F$ for Fourier isomorphism as in (\ref{diagpush})). In the following diagram, for keep short the notations, we only put the groupoid involved instead of  its 
twisted crossed product K-theory. With this convention understood, we need to prove that the following  diagram  is commutative.

\begin{equation}
{\tiny \xymatrix{
&&&&&\\
&&&&&
\\
&&f^*T^vN\rtimes T^vP\ar@{}[rrdd]_-{{\bf V}}\ar[rr]^-{D_f}\ar@{.>}[dd]&&N\times_M(P\times_M P)\ar[r]^-m\ar[d]_-{\mathscr{T}^m}\ar@{}[rd]|-{{\bf IV}}&
N\ar@/^3pc/[ddddd]^-{g_!}\ar[d]^-{\mathscr{T}}
\\
&T^vP\ar[ru]^-{\mathscr{T}_F}&&&T^vN \times_M(P\times_M P)\ar@{}[rd]|-{{\bf III}}\ar[d]_-{\mathscr{T}_F^m}\ar[r]^-m&T^vN\ar[d]^-{\mathscr{T}_F}
\\
P\ar@/^8pc/[rrrrruu]^-{f_!}\ar@/_8pc/[rrrrrddd]_-{(g\circ f)!}\ar@{}[rr]|-{{\bf VIII}}\ar[ru]_-{\mathscr{T}}\ar[rd]^-{\mathscr{T}}&&?\ar@{}[rrdd]|-{{\bf VI}}\ar@{.>}[dd]\ar@{.>}[rr]&&(g^*T^vL\rtimes T^vN)\times_M (P\times_M P)\ar@{}[rdd]|-{{\bf II}}\ar[r]_-m\ar[dd]_-{D^m_g}&g^*T^vL\rtimes T^vN\ar[dd]^-{D_g}
\\
&T^vP\ar[rdd]_-{\mathscr{T}_F}&&&&
\\
&&??\ar@{}[rrd]|-{{\bf VII}}\ar@{.>}[rr]\ar[d]^-m&& L\times_M(N\times_M N)\times_M (P\times_M P)
\ar@{}[rd]|-{{\bf I}}\ar[r]_-m\ar[d]^-m&L\times_M(N\times_M N)\ar[d]^-m
\\
&&f^*g^*T^vL\rtimes T^vP\ar[rr]_-{D_{g\circ f}}&& L\times_M (P\times_M P)\ar[r]_-m&L\\
&&&&&
}}
\end{equation}
As visually sketched in the diagram, we will separate it in 8 diagrams (I-VIII above). We will then prove that each of them is commutative. 

In the diagrams {\bf I, II, III} and {\bf IV} the notation $m$ stands for the isomorphisms in K-theory induced by the canonical Morita equivalences, it is immediate these diagrams commute by the naturality of the Thom isomorphisms and  the naturality of evaluation morphisms.

{\bf More NOTATION:} As we already mentioned above, every time we put $\alpha$ we mean the pullback twisting induced by $\alpha$ on the correspondent crossed product groupoid. We will also be dealing with several twistings coming this time from oriented vector bundles, and again, for keeping the notation as short as possible we will only denote by $\o_N$ the orientation twistings associated to the vector bundle $T^vN$ (similarly $\o_L$ corresponds to $T^vL$) independently of the  crossed product groupoid over which the twisting lives. The context is clear enough to understand that we are in fact using the twisting of some pullback vector bundle, for example over $P\rtimes \gr$ the twisting $\o_N$ corresponds in fact to the twisting coming from the $\gr$-vector bundle $f^*T^vN$ over $P$.

{\bf Definition and commutativity of diagram V.}

Consider the twisted equivariant Thom isomorphism associated to the vector bundle $f^*T^vN\oplus f^*T^vN$ over $f^*T^vN$ seen as a $f^*T^vN\rtimes (T^vP\rtimes \gr)$-vector bundle:
\begin{equation}
K_{\alpha+\o_g}^\gr(f^*T^vN\rtimes T^vP)\stackrel{\mathscr{T}}{\longrightarrow}K_{\alpha+\o_g+\o_N}^\gr((f^*T^vN\oplus f^*T^vN)\rtimes T^vP).
\end{equation}

Next, consider also the action of the groupoid $(f^*T^vN\oplus f^*T^vN)\rtimes T^vP$ on the vector bundle $\mathcal{L}^*$ over $f^*T^vN$ given by the pullback of $(T^vL)^*\to L$ by the canonical map $f^*T^vN\to L$. We can consider the associated transformation groupoid and the correspondent twisted equivariant Thom isomorphism
{\small
\begin{equation}\label{diagVa}
K_{\alpha+\o_g+\o_N}^\gr((f^*T^vN\oplus f^*T^vN)\rtimes T^vP)\stackrel{\mathscr{T}}{\longrightarrow}K_{\alpha+\o_g+\o_N+\o_L}^\gr(\mathcal{L}^*\rtimes ((f^*T^vN\oplus f^*T^vN)\rtimes T^vP)).
\end{equation}
}
We need now to consider the isomorphism induced by Fourier isomorphism:
{\small
\begin{equation}\label{diagVb}
K_{\alpha+\o_g+\o_N+\o_L}^\gr(\mathcal{L}^*\rtimes ((f^*T^vN\oplus f^*T^vN)\rtimes T^vP))\stackrel{\mathscr{F}}{\longrightarrow}K_{\alpha+\o_g+\o_N+\o_L}^\gr(\mathcal{L}\rtimes ((f^*T^vN\oplus f^*T^vN)\rtimes T^vP))
\end{equation}
}
to finally consider the composition of the two precedent morphisms:
{\small
\begin{equation}\label{diagVc}
K_{\alpha+\o_g+\o_N}^\gr((f^*T^vN\oplus f^*T^vN)\rtimes T^vP)\stackrel{\mathscr{T}_F}{\longrightarrow}K_{\alpha+\o_g+\o_N+\o_L}^\gr(\mathcal{L}\rtimes ((f^*T^vN\oplus f^*T^vN)\rtimes T^vP)).
\end{equation}
}

Remember that the deformation index morphism $\Dnc_f$ is defined using the deformation groupoid 
$$\gr_f\rightrightarrows \gr_{f}^{(0)}$$
where $\gr_{f}^{(0)}=f^*T^vN\times \{0\}\bigsqcup N\times_M P\times (0,1]$, the deformation to the normal cone associated to the immersion $P\stackrel{f\times id_P}{\rightarrow}N\times_M P$. We will consider a vector bundle over $\gr_{f}^{(0)}$: take $\Dnc_{T^vN}$ to be the deformation to the normal cone of 
$P\stackrel{s_0(f)\times id_P}{\rightarrow}T^vN\times_M P$, where $s_0:N\to T^vN$ stands for the zero section. We have a vector bundle
\begin{equation}
\Dnc_{T^vN}:=f^*T^vN\oplus f^*T^vN\bigsqcup T^vN\times_M P\times (0,1]\stackrel{\Dnc(\pi)}{\longrightarrow} f^*T^vN\bigsqcup N\times_M P\times (0,1],
\end{equation}
where $\Dnc(\pi)$ is the deformation of the morphisms of pairs
$\pi:(T^vN\times_M P,P)\to (N\times_M P,P)$.

Now, the groupoid $\gr_f$ acts on $\Dnc_{T^vN}$. Indeed we take  the deformation of the trivial action of $N\times_M (P\times_M P)$ on $T^vN\times_M P$. We can then consider the twisted equivariant Thom isomorphism 
\begin{equation}
K_{\alpha+\o_g}^\gr(\gr_f)\stackrel{\mathscr{T}_F}{\longrightarrow}K_{\alpha+\o_g+\o_{\Dnc_{T^vN}}}(\Dnc_{T^vN}\rtimes\gr_f)
\end{equation}
Notice that by construction, $\Dnc_{T^vN}\rtimes\gr_f$ is a groupoid with units $\gr_{f}^{(0)}=f^*T^vN\bigsqcup N\times_M P\times (0,1]$ and a deformation groupoid over $[0,1]$ with fibers 
$(f^*T^vN\oplus f^*T^vN)\rtimes T^vP$ at zero and $T^vN \times_M (P\times_M P)$ over $(0, 1]$.  There is then the associated deformation index:
\begin{equation}\label{line2diagV}
\xymatrix{
K_{\alpha+\o_g+\o_N}^\gr((f^*T^vN\oplus f^*T^vN)\rtimes T^vP)\ar[rdd]&K_{\alpha+\o_g+\o_{\Dnc_{T^vN}}}^\gr(\Dnc_{T^vN}\rtimes\gr_f)\ar[dd]^-{e_1}\ar[l]^-{e_0}_-{\cong}\\&\\
&K_{\alpha+\o_g+\o_N}^\gr(T^vN \times_M (P\times_M P))
}
\end{equation}
For overcome to diagram V, let us consider the map of couples: (remember $P\hookrightarrow N\times_M P$ with $f\times id_P$)
$$(N\times_M P,P)\to (L,L)$$
It induces a map between the deformations
$$\gr_{f}^{(0)}=f^*T^vN\bigsqcup N\times_M P\times (0,1]\to L\times [0,1].$$
We take the pullback of the vector bundle $T^vL\times [0,1]$ over $L\times [0,1]$ by this map, we denote it by $\mathcal{DL}\to \gr_{f}^{(0)}$. There is a canonical action of the  semi-direct product groupoid $\Dnc_{T^vN}\rtimes\gr_f$ on $\mathcal{DL}$, thus giving the respective twisted equivariant Thom isomorphism (modulo Fourier isomorphism  as \ref{diagVa}, \ref{diagVb} and \ref{diagVc} above)
\begin{equation}
K_{\alpha+\o_g+\o_{\Dnc_{T^vN}}}^\gr(\Dnc_{T^vN}\rtimes\gr_f)\stackrel{\mathscr{T}_F}{\longrightarrow}K_{\alpha+\o_g+\o_{\Dnc_{T^vN}}+\o_{\mathcal{DL}}}^\gr(\mathcal{DL}\rtimes(\Dnc_{T^vN}\rtimes\gr_f))
\end{equation}
By construction, $\mathcal{DL}\rtimes(\Dnc_{T^vN}\rtimes\gr_f)$ is a deformation groupoid over $[0,1]$ with fibers 
$\mathcal{L}\rtimes((f^*T^vN\oplus f^*T^vN)\rtimes T^vP)$ at zero and $g^*T^vL\rtimes (T^vN \times_M (P\times_M P))$ out of zero. There is then the associated deformation index:
\begin{equation}\label{line3diagV}
{\tiny
\xymatrix{
K_{\alpha+\o_g+\o_N+\o_L}^\gr(\mathcal{L}\rtimes((f^*T^vN\oplus f^*T^vN)\rtimes T^vP))\ar[rd]&K_{\alpha+\o_g+\o_{\Dnc_{T^vN}}+\o_{\mathcal{DL}}}^\gr(\mathcal{DL}\rtimes(\Dnc_{T^vN}\rtimes\gr_f))\ar[d]^-{e_1}\ar[l]^-{e_0}_-{\cong}\\
&K_{\alpha+\o_g+\o_N+\o_L}^\gr(g^*T^vL\rtimes (T^vN \times_M (P\times_M P)))
}}
\end{equation}

The diagram {\bf V} looks like:

\begin{equation}\label{diagV}
{\tiny 
\xymatrix{
K_{\alpha+\o_g}^\gr(f^*T^vN\rtimes T^vP)\ar[d]_-{\mathscr{T}}\ar@/^2pc/[rr]^-{D_f}&K_{\alpha+\o_g}^\gr(\gr_f)\ar[d]_-{\mathscr{T}}\ar[r]^-{e_1}\ar[l]^-{e_0}_-{\cong}&K_{\alpha+\o_g}^\gr(N\times_M (P\times_M P))\ar[d]^-{\mathscr{T}^m}
\\
K_{\alpha+\o_L}^\gr((f^*T^vN\oplus f^*T^vN))\rtimes T^vP\ar[d]_-{\mathscr{T}_F}&K_{\alpha+\o_L}^\gr(\Dnc_{T^vN}\rtimes\gr_f)\ar[d]_-{\mathscr{T}_F}\ar[r]^-{e_1}\ar[l]^-{e_0}_-{\cong}&
K_{\alpha+\o_L}^\gr(T^vN \times_M (P\times_M P))\ar[d]^-{\mathscr{T}_F^m}
\\
K_{\alpha}^\gr(\mathcal{L}\rtimes((f^*T^vN\oplus f^*T^vN)\rtimes T^vP))&K_{\alpha}^\gr(\mathcal{DL}\rtimes(\Dnc_{T^vN}\rtimes\gr_f))\ar[r]^-{e_1}\ar[l]^-{e_0}_-{\cong}&K_{\alpha}^\gr(g^*T^vL\rtimes T^vN)\times_M(P\times_M P))
}}
\end{equation}
where we have made a simplification of twistings: In the second line the twistings should be in principle those appearing in (\ref{line2diagV}), but notice that the canonical K-orientation of the vector bundle $f^*T^vN\oplus f^*T^vN$ induces an equivalence of twistings between $\o_g+\o_N$ and $\o_L$. Also, in the third line the twistings should be in principle those appearing in (\ref{line3diagV}), but for the same reason as before, $\o_g+\o_N+\o_L$ is canonically trivial.
 
The diagram above is evidently commutative by naturality with respect to morphisms of the twisted equivariant Thom isomorphism.

{\bf Definition and commutativity of diagram VI.}

The first groupoid we need to consider here is the Thom groupoid (\cite{DLN} theorem 6.2) associated to the real vector bundle $f^*T^vN$ over $P$, it consists on taking the tangent groupoid of the fiber product groupoid $f^*T^vN\times_Pf^*T^vN\rightrightarrows f^*T^vN$, it is then given by the deformation groupoid
$$\mathbf{T}_N:=f^*T^vN\oplus f^*T^vN \bigsqcup f^*T^vN\times_Pf^*T^vN\times (0,1]\rightrightarrows f^*T^vN\times [0,1].$$
The groupoid $T^vP$ acts (diagonally) on the Thom groupoid $\mathbf{T}_N$, we consider the semi-direct product groupoid 
$\mathbf{T}_N\rtimes T^vP$. We have as well a crossed product $
(\mathcal{L}\times [0,1])\rtimes (\mathbf{T}_N\rtimes T^vP)$. We can consider the associated  deformation index
\begin{equation}
{\small
\xymatrix{
K_{\alpha}^\gr(\mathcal{L}\rtimes((f^*T^vN\oplus f^*T^vN)\rtimes T^vP))\ar[rdd]&K_{\alpha}^\gr((\mathcal{L}\times [0,1])\rtimes(\mathbf{T}_N\rtimes T^vP))\ar[dd]_-{e_1}
\ar[l]^-{e_0}_-{\cong}\\ &\\
&K_{\alpha}^\gr(\mathcal{L}\rtimes((f^*T^vN\times_P f^*T^vN)\rtimes T^vP))
}}
\end{equation}
Next, consider the following immersion of groupoids
$$(g\circ f)\times f^2\times \Delta :P\rightarrow L\times_M(N\times_M N)\times(P\times_M P),$$
it gives as well a deformation groupoid
$\gr_{(g\circ f,f^2)}$ that induces a deformation index
\begin{equation}
{\small
\xymatrix{
K_\alpha^\gr(\mathcal{L}\rtimes((f^*T^vN\times_P f^*T^vN)\rtimes T^vP))&K_\alpha^\gr(\gr_{(g\circ f,f^2)})\ar[r]_-{e_1}
\ar[l]^-{e_0}_-{\cong}&K_\alpha^\gr(L\times_M(N\times_M N)\times(P\times_M P)).
}}
\end{equation}

Now we consider the deformation groupoid

$$\mathbb{D}:=\mathcal{DL}\rtimes(\Dnc_{T^vN}\rtimes\gr_f)\bigsqcup \gr_{(g\circ f,f^2)}\times (0,1].$$
The fact that the zero component of $\mathcal{DL}\rtimes(\Dnc_{T^vN}\rtimes\gr_f)$, that is $\mathcal{L}\rtimes((f^*T^vN\oplus f^*T^vN)\rtimes T^vP)$, can be glued (by the Lie groupoid $(\mathcal{L}\times [0,1])\rtimes(\mathbf{T}_N\rtimes T^vP)$) with the zero component of $\gr_{(g\circ f,f^2)}$, and the same for any $t\neq 0$ (glued for any such $t$ by the Lie groupoid $\gr_g^m$), tells us that there is a Lie groupoid structure over $\mathbb{D}$ compatible with the smooth structures of the departing groupoids (See \cite{Deb}, or \cite{DLN} section 2 for more details on smooth structures on deformation groupoids).

The diagram {\bf VI} follows from  the following commutative diagram:

\begin{equation}\label{diagVI}
{\small 
\xymatrix{
K_\alpha^\gr(\mathcal{L}\rtimes((f^*T^vN\oplus f^*T^vN)\rtimes T^vP))&K_\alpha^\gr(\mathcal{DL}\rtimes(\Dnc_{T^vN}\rtimes\gr_f))\ar[r]^-{e_1}\ar[l]^-{e_0}_-{\cong}&K_\alpha^\gr((g^*T^vL\rtimes T^vN)\times_M(P\times_M P))\ar@/^3pc/[dd]^-{D_g^m}
\\
K_\alpha^\gr((\mathcal{L}\times [0,1])\rtimes(\mathbf{T}_N\rtimes T^vP))\ar[d]_-{e_1}
\ar[u]^-{e_0}_-{\cong}&K_\alpha^\gr(\mathbb{D})\ar[d]_-{e_1}
\ar[u]^-{e_0}_-{\cong}\ar[r]^-{e_1}\ar[l]^-{e_0}_-{\cong}& K_\alpha^\gr(\gr_g^m)\ar[d]_-{e_1}
\ar[u]^-{e_0}_-{\cong}
\\
K_\alpha^\gr(\mathcal{L}\rtimes((f^*T^vN\times_P f^*T^vN)\rtimes T^vP))
&K_\alpha^\gr(\gr_{(g\circ f,f^2)})\ar[r]^-{e_1}\ar[l]^-{e_0}_-{\cong}&K_\alpha^\gr(L\times_M(N\times_M N)\times_M(P\times_M P))
}}
\end{equation}

{\bf Definition and commutativity of diagram VII.}

The canonical projection of couples
$$
\xymatrix{
L\times_M(N\times_M N)\times_M (P\times_M P)\ar[r]^-\pi&L\times_M (P\times_M P)
\\
&&\\
P\ar[uu]^-{(g\circ f)\times f^2\times \Delta}\ar[r]_-=&P\ar[uu]_-{g\circ f\times \Delta}
}
$$
induces a canonical projection between the deformations groupoids
$$\gr_{(g\circ f,f^2)}\to \gr_{g\circ f}$$ which is a Morita equivalence of groupoids. Fiberwise, the above projection corresponds to the Morita equivalence
$$f^*T^vN\times_Pf^*T^vN\to P$$ at zero, and
$$N\times_M N\to M$$ out of zero.

We have the induced isomorphism in K-theory and the following commutative diagram:

\begin{equation}\label{diagVII}
{\small 
\xymatrix{
K_\alpha^\gr(\mathcal{L}\rtimes((f^*T^vN\times_P f^*T^vN)\rtimes T^vP))\ar[dd]_-m
&K_\alpha^\gr(\gr_{(g\circ f,f^2)})\ar[r]^-{e_1}\ar[l]^-{e_0}_-{\cong}\ar[dd]_-m&K_\alpha^\gr(L\times_M(N\times_M N)\times_M(P\times_M P))\ar[dd]_-m
\\ &&
\\
K_\alpha^\gr(f^*g^*T^vL\rtimes T^vP)
\ar@/_2pc/[rr]_-{D_{g\circ f}}&K_\alpha^\gr(\gr_{g\circ f})\ar[r]^-{e_1}\ar[l]^-{e_0}_-{\cong}&K_\alpha^\gr(L\times_M(P\times_M P))
}}
\end{equation}

{\bf Commutativity of diagram VIII.}

This diagram looks as follows:
\begin{equation}\label{diagVIII}
{\tiny
\xymatrix{
&&&&K_{\alpha+\o_g}^\gr(f^*T^vN\rtimes T^vP)\ar[d]^-{\mathscr{T}}&\\
&K_{\alpha+\o_g+\o_N}^\gr(T^vP)\ar@{}[rrd]^-{{\bf A}}\ar[urrr]^-{\mathscr{T}_F}&&&K_{\alpha+\o_g+\o_N}^\gr((f^*T^vN\oplus f^*T^vN)\rtimes T^vP)\ar[d]^-{\mathscr{T}_F}&\\
K_{\alpha+\o_{g\circ f}}^\gr(P)\ar@{.>}[rrrr]\ar[ur]^-{\mathscr{T}}\ar[dr]_-{\mathscr{T}}&&&&K_\alpha^\gr(\mathcal{L}\rtimes((f^*T^vN\oplus f^*T^vN)\rtimes T^vP))&\\
&K_{\alpha+\o_L}^\gr(T^vP)\ar@{}[rrr]^-{{\bf B}}\ar[dddrrr]_-{\mathscr{T}_F}&&&K_\alpha^\gr((\mathcal{L}\times [0,1])\rtimes(\mathbf{T}_N\rtimes T^vP))\ar[d]^-{e_1}\ar[u]_-{e_0}^-{\cong}&\\
&&&&K_\alpha^\gr(f^*g^*T^vL\rtimes((f^*T^vN\times_P f^*T^vN)\rtimes T^vP))\ar[dd]^-m&\\
&&&&&\\
&&&&K_\alpha^\gr(f^*g^*T^vL\rtimes T^vP)&
}}
\end{equation}
where the morphisms $\mathscr{T}$ are Thom isomorphisms (with a subscript  $F$  if it is modulo Fourier isomorphism as before). 
As visually sketched above we will separate diagram VIII in two diagrams A and B. By proposition \ref{Thomproperties} the arrow that fits the pointed arrow above and that make diagram A commutative is 
\[
{\tiny 
\xymatrix{
K_{\alpha+\o_{g\circ f}}^\gr(P)\ar[r]^-{\mathscr{T}_0}&K_{\alpha}^\gr((f^*g^*T^vL)^*\oplus 
((f^*T^vN)^*\oplus f^*T^vN)\oplus T^vP)\ar[r]^-{\sigma_0}&K_\alpha^\gr((f^*g^*T^vL)^*
\rtimes(((f^*T^vN)^*\oplus f^*T^vN)\rtimes T^vP))\ar[d]^{\mathscr{F}}\\&&K_\alpha^\gr(\mathcal{L}
\rtimes((f^*T^vN^*\oplus f^*T^vN)\rtimes T^vP))
}}
\]
where $\mathscr{T}_0$ is the $\gr$-equivariant Thom isomorphism associated to $$(f^*g^*T^vL)^*\oplus 
((f^*T^vN)^*\oplus f^*T^vN)\oplus T^vP\to P,$$,  $\sigma_0 $ is the 
$\gr$-deformation index of the groupoid $(f^*g^*T^vL)^*\rtimes
((f^*T^vN)^*\oplus f^*T^vN)\rtimes T^vP$ and $\mathscr{F}$ is induced from the obvious Fourier isomorphism. We will denote by $\sigma_{0,F}$ the composition $\mathscr{F}\circ \sigma_0$. For diagram B we have the following decomposition into commutative subdiagrams:
\[
{\tiny
\xymatrix{
K_{\alpha+\o_{g\circ f}}^\gr(P)\ar[rr]^-{\mathscr{T}_0}&&K_{\alpha}^\gr((f^*g^*T^vL)^*\oplus 
((f^*T^vN)^*\oplus f^*T^vN)\oplus T^vP)\ar[rr]^-{\sigma_{0,F}}&&K_\alpha^\gr(\mathcal{L}\rtimes((f^*T^vN\oplus f^*T^vN)\rtimes T^vP))\\
K_{\alpha+\o_{g\circ f}}^\gr(P\times [0,1])\ar[rr]^-{\mathscr{T}_0^{[0,1]}}\ar[d]_-{e_1}^-{\cong}\ar[u]^-{e_0}_-{\cong}&&K_{\alpha}^\gr(((f^*g^*T^vL)^*\oplus ((f^*T^vN)^*\oplus f^*T^vN)\oplus T^vP)\times [0,1])\ar[rr]^-{\sigma_{\mathbf{T},F}}\ar[d]_-{e_1}^-{\cong}\ar[u]^-{e_0}_-{\cong}&&K_\alpha^\gr((\mathcal{L}\times [0,1])\rtimes(\mathbf{T}_N\rtimes T^vP))\ar[d]^-{e_1}\ar[u]_-{e_0}^-{\cong}\\
K_{\alpha+\o_{g\circ f}}^\gr(P)\ar[rr]^-{\mathscr{T}_0}\ar[d]_-{Id}&&K_{\alpha}^\gr((f^*g^*T^vL)^*\oplus ((f^*T^vN)^*\oplus f^*T^vN)\oplus T^vP)\ar[rr]^-{\sigma_{1,F}}\ar[d]_-{\sigma_m}&&K_\alpha^\gr(\mathcal{L}\rtimes((f^*T^vN\times_P f^*T^vN)\rtimes T^vP))\ar[d]^-m\\
K_{\alpha+\o_{g\circ f}}^\gr(P)\ar[rr]_-{\mathscr{T}}&&K_{\alpha}^\gr((f^*g^*T^vL)^*\oplus T^vP)\ar[rr]_-{\sigma_F}&&K_\alpha^\gr(f^*g^*T^vL\rtimes T^vP)
}}
\]
where 
\begin{itemize}
\item $\sigma_\mathbf{T}$ is the $\gr$-deformation index associated to the groupoid $(f^*g^*T^vL)^*\rtimes(\mathbf{T}_N\rtimes T^vP)$. In particular it commutes with the respective $\gr$-deformation indices corresponding to the evaluations at zero and one ($\sigma_0$ and $\sigma_1$). The subscript  
$\mathscr{F}$ above indicates modulo Fourier isomorphism.
\item $\mathscr{T}_0^{[0,1]}$ is the Thom isomorphism associated to the $\gr$-vector bundle $(f^*g^*T^vL\oplus (f^*T^vN\oplus f^*T^vN)\oplus T^vP)\times [0,1] $ over $P\times [0,1]$. In K-theory the evaluations ($e_0$ and $e_1$) from this morphism give the the same morphism $\mathscr{T}_0$.
\item $\sigma_m$ is the composition of the $\gr$-deformation index of the groupoid 
$(f^*g^*T^vL)^*\oplus (f^*T^vN\times_P f^*T^vN)\oplus T^vP$ followed by the morphism induced by the Morita equivalence $(f^*g^*T^vL)^*\oplus (f^*T^vN\times_P f^*T^vN)\oplus T^vP\to (f^*g^*T^vL)^*\oplus T^vP$. The commutativity of the right bottom square is then immediate by construction of the deformation indices. The commutativity of the left bottom square follows from proposition \ref{Thomproperties}, property 3.
\end{itemize}
For finish just let us remark that the Fourier isomorphism part of diagram B above obviously commute with evaluations. Diagram VIII is hence commutative and this ends the proof of the theorem.
\end{proof}

\section{Twisted geometric K-homology}

\subsection{Definition and some computations}

\begin{definition}[Twisted geometric K-homology]
Let $\gr\rightrightarrows M$ be a Lie groupoid with a twisting 
$\alpha:\gr--->PU(H)$. The {\em twisted geometric K-homology group}  associated to 
$(\gr,\alpha)$ is the abelian group denoted by $K_{*}^{geo}(\gr,\alpha)$ with generators given by the cycles $(P,x)$ where
\begin{itemize}
\item[(1)]   $P$ is a smooth  co-compact $\gr$-proper manifold,
\item[(2)]   $\pi_P:P\to M$ is  the smooth momentum map which supposed to be an oriented  submersion, and
\item[(3)] $x\in K^*(P\rtimes \gr, \pi_P^*\alpha+\o_{T^vP})$,
\end{itemize}
and the equivalence relations 
\begin{equation}
(P,x)\sim (P',g_!(x))
\end{equation}
where  $g:P\to P'$ is a smooth $\gr$-equivariant map (in particular $\pi_{P'}\circ g=\pi_P$).
\end{definition}

Next, we perform a computation in an explicit case:

\begin{proposition}\label{propQfinal}
Let $(\gr,\alpha)$ be a twisted groupoid, and let $\mathcal{C}_\gr$ be the category of proper $\gr$-manifolds as above and homotopy classes of smooth 
$\gr$-equivariant maps. Then if $Q$ is a final object for $\mathcal{C}_\gr$ one has an isomorphism
\begin{equation}
\xymatrix{
K_{*}^{geo}(\gr,\alpha)\ar[r]^-{\mu_Q}_-{\cong}& K^*(Q\rtimes \gr, \pi_Q^*\alpha+\o_{T^vQ}).
}
\end{equation}
\end{proposition}

\begin{proof}
Let $(P,x)$ be a geometric cycle over $(\gr,\alpha)$. By hypothesis there is a $\gr$-equivariant map $q_P:P\to Q$ since $Q$ is a final object in $\mathcal{C}_\gr$. We define $\mu_Q$ by
\begin{equation}
\mu_Q([P,x])=(q_Q)_!(x)
\end{equation}
which is well defined by theorem \ref{twistfun} above.

We will explicitly define the inverse. Let $y\in K^*(Q\rtimes \gr,\pi_Q^*\alpha+\o_{T^vQ})$, we define $\beta_Q(y)\in K_{*}^{geo}(\gr,\alpha)$ to be the class of the cycle $(Q,y)$.

In one direction, $\mu_Q(\beta_Q(y))=y$ is obvious, and in the  other direction, 
$\beta_Q(\mu_Q([P,x]))=[Q,(q_P)_!(x)]=[P,x]$.

\end{proof}

\begin{examples}\label{exfinalobj}
\item[(1)] The most basic example in which the last proposition applies is when the groupoid $\gr\rightrightarrows M$ is proper with $M/\gr$ compact. This  covers the  case of orbifold groupoids.
 Then we have  an explicit isomorphism
\begin{equation}
\xymatrix{
K_{*}^{geo}(\gr,\alpha)\ar[r]^-{\mu_Q}_-{\cong}& K^{*}(\gr, \alpha).
}
\end{equation}
\item[(2)] A very interesting example where one can apply the computation above is the following (Connes book \cite{Concg} $10.\beta$): Let $G$ be a connected Lie group and $\alpha:G\to PU(H)$ a projective representation. Let $L$ be a maximal compact subgroup of $G$, by a result of Abels and Borel (\cite{Ab}), the homogeneous space $Q=L\setminus G$ is a final object of $\mathcal{C}_G$. Then there is an explicit isomorphism
\begin{equation}\label{muhomG/L}
\xymatrix{
K^{geo}_{*}(G,\alpha)\ar[r]^-{\mu_{L\setminus G}}_-{\cong}&K^*((L\setminus G)\rtimes G,p^*\alpha+\o_{T_e(L\setminus G)})
}
\end{equation}
where $p: (L\setminus G) \rtimes G\to G$ is the canonical projection.
Note that the action of $G$ on the homogeneous space $L\setminus G$ is transitive, hence the groupoid $L\setminus G\rtimes G$ is transitive as well. Now, we know (proposition 5.14 in \cite{MM}) transitive groupoids are Morita equivalent to Lie groups, and more explicitly one Lie group model could be given by an isotropy group. In our case, it is easy to check that the isotopry group of the class of the identity $(L\setminus G\rtimes G)_{[e]}$ identifies canonically with $L$. Hence, the canonical inclusion $L\hookrightarrow L\setminus G\rtimes G$ given by
$$l\mapsto ([e],l)$$
is a Morita equivalence of groupoids (proposition 5.14 {\it (iv)} in \cite{MM}). Using (\ref{muhomG/L}) and the Morita equivalence just described, we can obtain an isomorphism
\begin{equation}\label{muL}
\xymatrix{
K^{geo}_{*}(G,\alpha)\ar[r]^-{\mu_L}_-{\cong}&K^*(L,i^*\alpha+\o_{T_e(L\setminus G)})
}
\end{equation}
where $i:L\hookrightarrow G$ is the inclusion and where $T_e(L\setminus G)$ is considered as a $L$-vector space. Notice that when $T_e(L\setminus G)$ is even dimensional and the isotropy representation of $L$ on this space lifts to $Spin(T_e(L\setminus G))$, one has that $\o_{T_e(L\setminus G)}$ is equivalent to the trivial twisting. In particular if $\alpha$ is also trivial the right hand side of (\ref{muL}) above is isomorphic to $R(L)$, the representation ring of $L$.
\end{examples}

\subsection{Morita invariance}

\begin{theorem}[Morita invariance of the geometric K-homology]\label{MoritageoK}
Let $\gr$ and $\gr'$ be two Morita equivalent groupoids. Let us denote by $\gr\stackrel{\phi}{--->}\gr'$ the generalized isomorphism. Given  $\alpha':\gr'---> PU(H)$ a twisting, there is an induced isomorphism of groups
\begin{equation}
\xymatrix{
K^{geo}(\gr,\alpha)\ar[r]^-{\phi_*}_-{\cong}&K^{geo}(\gr',\alpha')
}
\end{equation}
where $\alpha:=\alpha'\circ\phi$ is the induced twisting on $\gr$.
\end{theorem}

\begin{proof}
First of all let us write the generalized isomorphism $\xymatrix{
\gr\ar@{.>}[r]^-{\phi}&\gr'}$ as 
\[
\xymatrix{
\gr \ar@<.5ex>[d]\ar@<-.5ex>[d]&P_\phi \ar@{->>}[ld]_-{t_\phi} \ar[rd]^-{s_\phi}&\gr' \ar@<.5ex>[d]\ar@<-.5ex>[d]\\
 &M&&M'.
}
\]
the Morita bi-bundle giving the Morita equivalence.

{\bf Step 1.  The definition of $\phi_*$:}  We will now describe the morphism $\phi_*$ at the geometric cycle level. Let $(P,x)$ be a geometric cycle over $(\gr,\alpha)$, we will let
\begin{equation}
\phi_*(P,x):=(\phi(P),\phi(x)),
\end{equation}
where
\begin{itemize}
\item $\phi(P):=P\times_\gr P_\phi=P\times_MP_\phi/(p,p')\sim$ with $\sim$ given by $ (p\cdot \gamma,\gamma^{-1}\cdot p')$. 

\item The fact that $\gr$ acts freely and properly on $P_\phi$ on the left and properly on $P$ on the right  implies that $P\times_\gr P_\phi$ has indeed an induced manifold structure. Now, the action of $\gr'$ on $P\times_\gr P_\phi$ with momentum map $\phi(\pi_P):=s_\phi\circ \pi_2$ is defined as:
$$[(p,p')]\cdot \gamma':=[p,p'\cdot \gamma'],$$
which is evidently well defined. Notice that the action is proper since the same is true for the action of $\gr'$ on $P_\phi$ but the action is free if and only if $\gr$ acts freely on $P$. Hence, 
$P\times_\gr P_\phi$ is a $\gr'$-proper manifold.
\item Letting $\pi_2:P\times_\gr P_\phi\to P_\phi$ the second projection, $\pi_{P\times_\gr P_\phi}:=s_\phi\circ \pi_2:P\times_\gr P_\phi\to M'$ is a smooth submersion since $s_\phi$ is also a submersion as $\phi$ is a generalized isomorphism.
\item The element $\phi(x)$: For this purpose, let us consider the inverse Morita bi-bundle
\begin{equation}\label{inversephi}
\xymatrix{
\gr'\ar@{.>}[r]^-{\phi^{-1}}&\gr:&\gr' \ar@<.5ex>[d]\ar@<-.5ex>[d]&P_{\phi^{-1}} \ar@{->>}[ld]_-{t_{\phi^{-1}}} \ar[rd]^-{s_{\phi^{-1}}}&\gr \ar@<.5ex>[d]\ar@<-.5ex>[d]\\
&&M'&&M.
}
\end{equation}
By definition $P_\phi\times_{\gr'}P_{\phi^{-1}}$ is equivalent to the $\gr$-bundle over $\gr$ associated to the identity $\gr\to\gr$ (which has as total space $\gr$ itself), and similarly $P_{\phi^{-1}}\times_{\gr}P_\phi$ is equivalent to $\gr'$ as $\gr'$-bundle over $\gr'$. As an immediate consequence we have the following two bi-bundles between the crossed product groupoids which induce generalized morphisms inverses of each other:
\begin{equation}\label{moritaPrtimesG}
\xymatrix{
P\rtimes\gr\ar@{.>}[r]^-{\phi^P}&(P\times_\gr P_\phi)\rtimes\gr':&P\rtimes\gr  \ar@<.5ex>[d]\ar@<-.5ex>[d]&P\times_M P_\phi \ar@{->>}[ld]_-{\pi_1} \ar[rd]^-{q}&(P\times_\gr P_\phi)\rtimes\gr' \ar@<.5ex>[d]\ar@<-.5ex>[d]\\
&&P&&P\times_\gr P_\phi.
}
\end{equation}
and
\begin{equation}\label{moritaPrtimesGinverse}
\xymatrix{
(P\times_\gr P_\phi)\rtimes\gr'\ar@{.>}[r]^-{(\phi^P)^{-1}}&P\rtimes\gr:&(P\times_\gr P_\phi)\rtimes\gr'  \ar@<.5ex>[d]\ar@<-.5ex>[d]&P\times_{M'} P_{\phi^{-1}} \ar@{->>}[ld]_-{q'} \ar[rd]^-{\pi_1}&P\rtimes\gr \ar@<.5ex>[d]\ar@<-.5ex>[d]\\
&&P\times_\gr P_\phi&&P.
}
\end{equation}
Here $q$ and $q'$ are the obvious projection maps. 
Notice now that we have by definition the following two commutative diagrams of generalized morphisms:
\begin{equation}
\xymatrix{
P\rtimes\gr \ar@{.>}[d]_-{\phi^P} \ar@{.>}[rr]^-{\pi_P}&&\gr\ar@{.>}[rr]^-{\alpha}\ar@{.>}[d]^-{\phi}&&PU(H)\\
(P\times_\gr P_\phi)\rtimes\gr'\ar@{.>}[rr]_-{\pi_{(P\times_\gr P_\phi)}}&&\gr'\ar@{.>}[rru]_-{\alpha'}&&
}
\end{equation}
and 
\begin{equation}
\xymatrix{
P\rtimes\gr \ar@{.>}[drr]_-{\o_{T^vP}}\ar@{.>}[rr]^-{\phi^P}&&(P\times_\gr P_\phi)\rtimes\gr'\ar@{.>}[d]^-{\o_{T^v(P\times_\gr P_\phi)}}\\ &&PU(H)
}
\end{equation}
\end{itemize}

Hence we have an induced Morita equivalence between the respective extensions:
\begin{equation}\label{extphi}
\xymatrix{
R_{\pi_P^*+\o_{T^vP}}\ar@{.>}[r]^-{\tilde{\phi}}& R_{\pi_{\phi(P)}^*\alpha'+\o_{T^v\phi(P)}}
}
\end{equation}

This defines a Morita equivalence between the respective $C^*$-algebras and since it is an equivalence of extensions it preserves in particular the $\mathbb{Z}$-grading (\ref{gradext}). We have then an associated element $\phi(x)\in K(\phi(P)\rtimes \gr',\alpha'+\o_{T^v\phi(P)}).$

\vspace{2mm}

{\bf Step 2.   $\phi_*$ is a well defined morphism:} Consider a $\gr-$equivariant map 
$g:P\to P'$, then by definition $(P,x)\sim (P',g_!(x))$. We let $\phi(g):\phi(P)\to \phi(P')$ the smooth map given by $\phi(g)[p,z]:=[g(p),z]$ that is well defined since $g$ is $\gr$-equivariant. It is also clear $\phi(g)$ is $\gr'$-equivariant. We have then the following commutative diagram of generalized morphisms:
\begin{equation}
\xymatrix{
P\rtimes \gr\ar[r]^-g\ar@{.>}[d]_-{\phi^P}&P'\rtimes \gr\ar@{.>}[d]^-{\phi^{P'}} \\
\phi(P)\rtimes \gr'\ar[r]_-{\phi(g)}&\phi(P')\rtimes \gr'
}
\end{equation}
from which we get that $$\phi(g)!(\phi(x))=\phi(g_!(x))$$
and hence $$\phi_*(P,x)\sim \phi_*(P',g_!(x)),$$
that is, $\phi_*$ is a well defined morphism from $K_{*}^{geo}(\gr,\alpha)$ to 
$K_{*}^{geo}(\gr',\alpha')$.

\vspace{2mm}

{\bf Step 3.  $\phi_*$ is an isomorphism:} Associated to the inverse Morita bi-bundle (\ref{inversephi}) we have an analogously defined morphism $(\phi^{-1})_*$. We will show this is the inverse of $\phi_*$. For this purpose it is enough to check it at the cycle level:
\begin{itemize}
\item By definition 
$$(P\times_\gr P_\phi)\times_{\gr'}P_{\phi^{-1}}\cong P\times_{\gr}(P_\phi \times_{\gr'}P_{\phi^{-1}})\cong P\times_\gr \gr\cong P$$
 as $\gr$-manifolds over $M$.
\item Also, we have
$$\phi^{-1}(\phi(x))=x$$
by using the inverse $\tilde{\phi}^{-1}$ of the extension morphism above (\ref{extphi}).
\end{itemize}
We have then $(\phi^{-1})_*\circ\phi_*=Id_{K_*^{geo}(\gr,\alpha)}$. In a symmetric way we easily verify as well that $\phi_*\circ (\phi^{-1})_*=Id_{K_*^{geo}(\gr',\alpha')}$.
\end{proof}

\section{The twisted Baum-Connes assembly map for Lie groupoids}

\subsection{The assembly map}
We are now ready to state and prove one of the main results of this paper. 
 
 \begin{theorem}\label{twistedassemblymap}
 Let $(P,x)$ be a geometric cycle over $(\gr,\alpha)$. Let $\mu_{\alpha}(P,x)=(\pi_P)_!(x)$ be the element   in $K^*(\gr,\alpha)$. Then
  $\mu_{\alpha}(P,x)$ only depends upon the equivalence class of the twisted cycle  $(P,x)$. Hence we have a well defined assembly map
\begin{equation}
\mu_{\alpha}:K_{*}^{geo}(\gr,\alpha)\longrightarrow K^*(\gr,\alpha). 
\end{equation} 
\end{theorem}

\begin{proof}
It follows from the functoriality for proper $\gr$-manifolds, theorem \ref{twistfun} above.
\end{proof}

The following result is an easy consequence of proposition \ref{propQfinal}:

\begin{corollary}
Let $(\gr,\alpha)$ be a twisted groupoid, and let $\mathcal{C}_\gr$ be the category of proper $\gr$-manifolds as above and homotopy classes of smooth 
$\gr$-equivariant maps. Then if $Q$ is a final object for $\mathcal{C}_\gr$ with momentum map $\pi_Q:Q\to M$, one has the following commutative diagram
\begin{equation}
\xymatrix{
K_{*}^{geo}(\gr,\alpha)\ar[rd]_-{\mu_\alpha}\ar[rr]^-{\mu_Q}_-{\cong}&& K^*(Q\rtimes \gr, \pi_Q^*\alpha+\o_{T^vQ})\ar[ld]^-{(\pi_Q)_!}\\
&K^*(\gr,\alpha)&
}
\end{equation}
\end{corollary}

We can discuss the consquence of the last corollary for the two examples treated in \ref{exfinalobj}:

\begin{examples}
\item[(1)] For the case of a proper groupoid $\gr\rightrightarrows M$ with $M/\gr$ compact we have an isomorphism of the assembly map. Indeed, in this case $M$ itself is a final object for $\mathcal{C}_\gr$ and the assembly becomes simply $\mu_M$ which was explicitly shown to be an isomorphism in proposition \ref{propQfinal}.
\item[(2)] Take again $G$ to be a connected Lie group and $\alpha:G\to PU(H)$ a projective representation. Let $L$ be a maximal compact subgroup of $G$. Putting together the assembly map and the discussion in \ref{exfinalobj} above, we have a commutative diagram
\begin{equation}
\xymatrix{
K_{*}^{geo}(G,\alpha)\ar[rd]_-{\mu_\alpha}\ar[rr]^-{\mu_L}_-{\cong}&& K^*(L, i^*\alpha+\o_{T_e(L\setminus G)})\ar[ld]^-{i_!}\\
&K^*(G,\alpha)&
}
\end{equation}
where $i:L\hookrightarrow G$ is the inclusion morphism. In the case $\alpha$ and $\o_{T_e(L\setminus G)}$ are trivial, the above diagram gives a meaning to Mackey's observations on unitary representations for Lie groups, at least in the case where the assembly map is an isomorphism. In the twisted case there should also be a relation between the projective representations of some Lie groups and certain related semi-direct product group's projective representations\footnote{By Thom isomorphism $K^*(L, i^*\alpha+\o_{T_e(L\setminus G)})\cong K^*(T_e(L\setminus G)\rtimes L,i^*\alpha)$}. We will leave this very interesting subject of study for further works.
\end{examples}

\subsection{Morita invariance of the assembly map}

\begin{theorem}[Morita invariance of the assembly map]\label{MoritaAS}
Let $\gr$ and $\gr'$ be two Morita equivalent groupoids. Let us denote by $\gr\stackrel{\phi}{--->}\gr'$ the generalized isomorphism (the Morita bi-bundle). Given  $\alpha':\gr'---> PU(H)$ a twisting, there is a commutative diagram
\begin{equation}
\xymatrix{
K^{geo}(\gr,\alpha)\ar[r]^-{\phi_*}_-{\cong}\ar[d]_-{\mu_\alpha}&K^{geo}(\gr',\alpha')\ar[d]^-{\mu_\alpha'}\\
K^*(\gr,\alpha)\ar[r]_-{\phi_*}^-{\cong}&K^*(\gr',\alpha')
}
\end{equation}
where $\alpha:=\alpha'\circ\phi$ is the induced twisting on $\gr$.
\end{theorem}

\begin{proof}
Let $(P,x)$ be a geometric cycle over $(\gr,\alpha)$. We will be using the notations and terminologies of theorem \ref{MoritageoK} and its proof. In particular see the induced geometric cycle $(\phi(P),\phi(x)))$ over $(\gr',\alpha')$ in the proof. 
It suffices  to prove that the following diagram is commutative
\begin{equation}
\xymatrix{
K^*(P\rtimes \gr,\alpha+\o_{T^vP})\ar@{}[rd]|-{{\bf I}}\ar[d]_-{\cT}^-{\cong}
\ar[r]^-{\phi_*}_-{\cong}&K^*(\phi(P)\rtimes \gr',\alpha'+\o_{T^v\phi(P)})\ar[d]_-{\cT}^-{\cong}
\\
K^*(T^vP\rtimes \gr,\alpha)\ar@{}[rd]|-{{\bf II}}\ar[r]^-{\phi_*}_-{\cong}&
K^*(T^v\phi(P)\rtimes \gr',\alpha')
\\
K^*(\gr_f\rtimes \gr,\alpha_f)\ar@{}[rd]|-{{\bf III}}
\ar[u]^-{e_0}_-{\cong}\ar[d]_-{e_1}
\ar[r]^-{\phi_*}_-{\cong}&K^*(\gr_{\phi(f)}\rtimes \gr',\alpha_{\phi(f)})\ar[u]^-{e_0}_-{\cong}\ar[d]_-{e_1}
\\
K^*((P\times_MP)\rtimes \gr,\alpha\circ\mu)\ar@{}[rd]|-{{\bf IV}}\ar[r]^-{\phi_*}_-{\cong}\ar[d]_-{\mu_*}^-{\cong}&K^*(
(\phi(P)\times_{M'}\phi(P))\rtimes \gr',\alpha'\circ \mu)\ar[d]_-{\mu_*}^-{\cong}
\\
K^*(\gr,\alpha)\ar[r]^-{\phi_*}_-{\cong}&K^*(\gr',\alpha')
}
\end{equation}
where we are denoting by $\phi_*$ the isomorphisms induced by the Morita equivalences coming naturally from $\phi$, and $\gr_f$ and  $\gr_{\phi(f)}$ are the deformation groupoids associated to $f=\pi_P: P\to M$ and $\phi(f)
= \pi_{P\times_\gr P_\phi}: \phi(P) \to M'$ respectively. We now describe them in a more explicit way as below. 

\begin{itemize}
\item Commutativity of diagram I above: The commutativity follows from \ref{Thomproperties} applied to $E=T^vP$.
\item Commutativity of diagram II and III above: we explicitly described in (\ref{moritaPrtimesG}) and (\ref{moritaPrtimesGinverse}) the Morita bi-bundle between $P\rtimes \gr$ and $\phi(P)\rtimes \gr'$ and in a complete analogous way it is possible to describe the Morita equivalences between $T^vP\rtimes \gr$ and $T^v\phi(P)\rtimes \gr'$, between $\gr_f\rtimes \gr$ and $\gr_{\phi(f)}\rtimes \gr'$ and between $(P\times_MP)\rtimes \gr$ and $\phi(P)\times_{M'}\phi(P))\rtimes \gr'$. In fact, the Morita bi-bundle between $\gr_f\rtimes \gr$ and $\gr_{\phi(f)}\rtimes \gr'$ is simply given by
\begin{equation}\label{moritaGfrtimesG}
\xymatrix{
\gr_f\rtimes\gr\ar@{.>}[r]^-{\phi_{\gr_f}}&(\gr_f\times_\gr P_\phi)\rtimes\gr':&\gr_f\rtimes\gr  \ar@<.5ex>[d]\ar@<-.5ex>[d]&\gr_f\times_M P_\phi \ar@{->>}[ld]_-{\pi_1} \ar[rd]^-{q}&(\gr_f\times_\gr P_\phi)\rtimes\gr' \ar@<.5ex>[d]\ar@<-.5ex>[d]\\
&&\gr_f&&\gr_f\times_\gr P_\phi.
}
\end{equation}
Exactly as (\ref{moritaPrtimesG}) and (\ref{moritaPrtimesGinverse}), $\phi_{\gr_f}$ is an invertible Hilsum-Skandalis morphism. By construction, it is compatible with the restrictions to $T^vP\rtimes \gr$ and to $(P\times_MP)\rtimes \gr$, in other words, we have the following two commutative diagrams of generalized morphisms:
\begin{equation}
\xymatrix{
\gr_f\rtimes \gr \ar@{.>}[r]^-{\phi_{\gr_f}}&\gr_{\phi(f)}\rtimes \gr'&and&\gr_f\rtimes \gr \ar@{.>}[r]^-{\phi_{\gr_f}}&\gr_{\phi(f)}\rtimes \gr'\\ 
T^vP\rtimes \gr \ar[u]^-{i_0}\ar@{.>}[r]_-{\phi_{T^vP}}&T^v\phi(P)\rtimes \gr'\ar[u]_-{i_0}&&(P\times_MP)\rtimes \gr \ar[u]^-{i_1}\ar@{.>}[r]_-{\phi_{P\times_MP}}&(\phi(P)\times_{M'}\phi(P))\rtimes \gr'\ar[u]_-{i_1},
}
\end{equation}
from which the commutativity of diagrams II and III follows immediately.
\item Commutativity of diagram IV above: the following diagram of generalized isomorphisms, where $\mu$ stand for the canonical projections, is commutative
\begin{equation}
\xymatrix{
(P\times_MP)\rtimes \gr \ar[d]_-{\mu}\ar@{.>}[rr]^-{\phi_{P\times_MP}}&&(\phi(P)\times_{M'}\phi(P))\rtimes \gr'\ar[d]_-{\mu}
\\
\gr \ar@{.>}[rr]_-{\phi}&&\gr'
}
\end{equation}
It implies the commutativity of diagram IV.
\end{itemize}

\end{proof}

\section{Comparison with the classic assembly maps}

\subsection{The twisted geometric assembly map as the $S^1$-invariant part of the "classic" geometric assembly map}

The definition of the twisted geometric K-homology groups is drawn from Connes definition (\cite{Concg} II.10.$\alpha$) for general Lie groupoids. 


Now, given a twisted Lie groupoid $(\gr,\alpha)$ we can consider the associated $S^1$-central extension $R_\alpha$ for which we have the geometric assembly map for the Lie groupoid $R_\alpha$, as a Lie groupoid with trivial twisting:
\begin{equation}
\mu_{R_\alpha}:K^{geo}(R_\alpha)\to K^*(R_\alpha)
\end{equation}

We have the following proposition:

\begin{proposition}\label{gradprop}
With the same notations as above we have an isomorphism of groups
\begin{equation}
\bigoplus_{n\in \mathbb{Z}}K^{geo}_*(\gr,n\alpha)\cong K^{geo}_{*}(R_\alpha)
\end{equation}
and under this isomorphism
\begin{equation}
\bigoplus_{n\in \mathbb{Z}}\mu_{n\alpha}= \mu_{R_\alpha}
\end{equation}
\end{proposition}

\begin{proof}
We will first describe a morphism
\begin{equation}
K^{geo}_*(\gr,n\alpha)\longrightarrow K^{geo}_{*}(R_\alpha).
\end{equation}
for every $n\in \mathbb{Z}$.
Let $[P,x]\in K^{geo}_*(\gr,n\alpha)$. By using the Thom isomorphism we might assume $\o_{T^vP}$ is a trivial twisting. Next consider the pullback diagram
$$\xymatrix{
P_\Omega\ar[r]\ar[d]&P\ar[d]^-f\\
\sqcup_{i\in I}\Omega_i\ar[r]&M
}
$$
that is, $P_\Omega=\{((x,i),p):f(p)=x\}$. We can consider $P_\Omega$ as a $R_\alpha$-manifold with the following action
$$((x,i),p)\cdot ((i,\gamma,j),u):=((s(\gamma),j),\gamma\cdot p),$$
where $\gamma\cdot p$ is the respective action of $\gamma\in \gr$ on $p\in P$ (for that we need $f(p)=t(\gamma)$). Because $P$ is a $\gr$-proper manifold it follows immediately that $P_\Omega$ is a $R_\alpha$-proper manifold. It is easy now to verify that the respective crossed product groupoid, $P_\Omega\rtimes R_\alpha$, can be identified as the $S^1$-central extension associated to the twisted groupoid $(P\rtimes\gr,\alpha\circ \pi_P)$, that is,  
$$P_\Omega\rtimes R_\alpha=R_{\alpha\circ \pi_P}.$$
Hence, 
$$K^*(P_\Omega\rtimes R_\alpha)\cong \bigoplus_nK^*(P\rtimes\gr,(\alpha\circ\pi_P)^n)$$ and we can associate to our $x\in K^*(P\rtimes\gr,n\alpha)$ the respective element in $K^*(P_\Omega\rtimes R_\alpha)$ which we denote also by $x$. We have then a natural morphism $[P,x]\mapsto [P_\Omega,x]$ from $K^{geo}_*(\gr,n\alpha)$ to $K^{geo}_{*}(R_\alpha)$ which is again well defined by wrong way functoriality. Thus, we obtain a morphism
\begin{equation}
\bigoplus_{n\in\mathbb{Z}}K^{geo}_*(\gr,n\alpha)\longrightarrow K^{geo}_{*}(R_\alpha).
\end{equation}
that satisfies by construction:
\begin{itemize}
\item It is injective and
\item it fits in the following commutative diagram
\begin{equation}
\xymatrix{
\bigoplus_{n\in\mathbb{Z}}K^{geo}_*(\gr,n\alpha)\ar[r]\ar[dr]_-{\bigoplus_{n\in \mathbb{Z}}\mu_{n\alpha}}&K^{geo}_{*}(R_\alpha)\ar[r]^-{\mu_{R_\alpha}}&K^*(R_\alpha)\ar[dl]^-{\cong}\\
&\bigoplus_{n\in\mathbb{Z}}K^*(\gr,n\alpha).&
}
\end{equation}
\end{itemize}
The surjectivity is as follows: Let $Z$ be a proper $R_\alpha$-manifold and $y\in K^*(Z\rtimes R_\alpha)$ (we can assume again, modulo the Thom isomorphism, $\o_{T^vZ}$ trivial). Consider the smooth manifold $X:=Z/S^1$ resulted from the canonical free and proper action of $S^1$ on $Z$ (explained for instance in \cite{TXL} section 2.2 page 850), there is an associated proper action of $\gr_\Omega$ on $X$ where $\Omega$ is the open cover associated with $\alpha$. Now, taking the canonical projection $X\stackrel{\pi}{\longrightarrow}\sqcup \Omega_i$ we can consider $P:=X/\sim$ with $x\sim y$ iff $\pi(x)=\pi(y)$, then $P$ is a smooth manifold, there is a projection $P\stackrel{\pi_P}{\longrightarrow}M$ and there is an induced proper action of $\gr$ on $P$. We can now easily identify the following two crossed product groupoids
$$X\rtimes \gr_\Omega\cong (P\rtimes \gr)_{\pi_P^{-1}(\Omega)}$$
and hence we also have an identification between the respective $S^1$-central extensions:
$$Z\rtimes R_\alpha\cong R_{\alpha\circ\pi_P}.$$
Thus, under these identifications, $K^*(Z\rtimes R_\alpha)=K^*(R_{\alpha\circ\pi_P})=\bigoplus_{n\in \mathbb{Z}}K^*(P\rtimes \gr, (\alpha\circ\pi_P)^n)$
from where the surjectivity follows.

\end{proof}

\begin{corollary}\label{corextiso}
$\mu_{n\alpha}$ is an isomorphism $\forall n\in \mathbb{Z}$ if and only if $\mu_{R_\alpha}$ is an isomorphism. In particular the geometric twisted assembly map is an isomorphism whenever the geometric assembly map for the correspondent extension is.
\end{corollary}

\subsection{Comparison with the analytic assembly map}
Until now we have not assumed our groupoids to be Hausdorff. For Hausdorff groupoids there is an analytic version of the assembly map that has been very productive in many applications, in particular thanks to the extensive use of Kasparov's KK-theory methods.

Let $R\rightrightarrows R_0$ be a Hausdorff Lie groupoid, we recall briefly the definition of its analytical K-homology group from  \cite{Tu3}:
\begin{equation}
K^{ana}_{*}(R):=\lim_{Y\subset {\bf E}R}KK_R^*(C_0(Y),C_0(R_0)).
\end{equation} 
Here ${\bf E}R$ is the universal space for proper $R$-actions. 

There is a canonical group morphism between the geometric and the analytical K-homology groups:
\begin{equation}
\xymatrix{
K^{geo}_{*}(R)\ar[r]^-{\lambda_R}&K^{ana}_{*}(R)
}
\end{equation}
that we will now explicitly describe: Let $[P,x]\in K^{geo}_{*}(R)$. We can construct an element
$(\pi_P)_!\in KK_R^*(T^vP,R_0)$ exactly as we did in section \ref{pushsection}. Now, by definition of ${\bf E}R$ there is a $Y\subset {\bf E}R$ and an element $c_P\in KK_R(Y,T^vP)$ induced by the classifying map $c:T^vP\to 
Y\subset {\bf E}R$. We set
\begin{equation}
\lambda_R([P,x])=[c_P\otimes_{T^vP}\pi_P]
\end{equation}
 
\begin{proposition}\label{assemblytopana}
We have the following commutative diagram:
\begin{equation}\label{710}
\xymatrix{
K^{geo}_{*}(R)\ar[rd]_-{\mu_{R}}\ar[rr]^-{\lambda_R}&&K^{ana}_{*}(R)
\ar[ld]^-{\mu_{R}^{ana}}\\
&K^*(R)&
}
\end{equation}
where $\mu_{R}^{ana}$ is the analytic assembly map, \cite{Tu3}.
We have moreover the Morita invariance of each morphism in (\ref{710}).
\end{proposition}


\subsection{Applications: Some cases where the geometric twisted assembly map is an isomorphism}
Still in the case of Hausdorff groupoids, proposition \ref{assemblytopana} implies the following:
\begin{corollary}\label{corassemblytopana}
If $\lambda_R:K^{geo}_{*}(R)\longrightarrow K^{ana}_{*}(R)$ is an isomorphism, then 
\begin{center}
$\mu_{R}$ is an isomorphism (resp. injective, resp. surjective) iff $\mu_{R}^{ana}$ is an isomorphism (resp. injective, resp. surjective).
\end{center}
\end{corollary}

\begin{examples}
Some examples of Lie groupoids for which the analytic assembly map is an isomorphism (or injective) are the following
\begin{enumerate}
\item injectivity for bolic groupoids (Tu \cite{Tu1})
\item isomorphism for groupoids having the Haagerup property (Tu \cite{Tu2})
\item isomorphism for almost connected Lie groups (Chabert-Echterhoff-Nest \cite{CEN})
\item isomorphism for hyperbolic groups (Lafforgue \cite{Laff12})
\end{enumerate}
\end{examples}
For the twisted case we put the last corollary together with corollary \ref{corextiso} to obtain:

\begin{corollary}\label{corcorassemblytopana}
Let $(\gr,\alpha)$ be a twisted (Hausdorff) Lie groupoid. Take $R_\alpha$ the corresponding $S^1$-central extension. Assuming $\lambda_{R_\alpha}:K^{geo}_{*}(R_\alpha)\longrightarrow K^{ana}_{*}(R_\alpha)$ is an isomorphism we have that the geometric twisted assembly map for $(\gr,\alpha)$ is an isomorphism whenever the analytic assembly map for $R_\alpha$ is.
\end{corollary}

\begin{example}
A very interesting example of the previous situation is when the groupoid 
$\gr$ satisfies the so called Haagerup property. Indeed, in this case, one can check that for any twisting $\alpha$, the correspondent extension groupoid $R_\alpha$ satisfies as well the Haagerup property. Then by Tu's theorem (\cite{Tu2} theorem 9.3, see also \cite{Tu3} theorem 6.1) the analytic assembly map for $R_\alpha$ is an isomorphism. This was already mentioned in Tu's habilitation \cite{Tu4} page 16.
Among the groupoids satisfying the Haagerup property one finds amenable groupoids.
\end{example}

A very interesting question then is the following one:

{\bf Question:} For which Lie groupoids is the comparison map between geometric and analytic K-homology an isomorphism?

In the twisted case the above question is even more precise: For which twisted Lie groupoids $(\gr,\alpha)$ is the comparison map $\lambda_{R_\alpha}$ an isomorphism?

Let us mention that different models for K-homology (at least in the untwisted case) were assumed by the experts to be isomorphic for many years. It was not until some years ago that a formal proof for some models was achieved (\cite{BHS,BOOSW}). In conclusion, the questions we are addressing are not trivial and, as we stated above, a positive answer have very interesting consequences. 

\appendix

\section{The twisted equivariant Thom isomorphism}

In this subsection we will establish  the Thom isomorphism in $\gr$-equivariant twisted K-theory which generalizes the non-equivariant twisted Thom isomorphism in \cite{CW}. We will need some basics on KK-theory:

\subsection{Hilsum-Skandalis-Le Gall descent functors and suspension maps on KK-theory}

In \cite{HS} section 2.1 Hilsum and Skandalis give a very explicit description of a group morphism
$$i^*:KK_H(A,B)\longrightarrow KK(A\rtimes_i \gr,B\rtimes_i \gr)$$
constructed from a groupoid cocycle
$$\gr\stackrel{i}{--->}H$$ for any $A,B$ $H$-algebras. The algebras $A\rtimes_i \gr,B\rtimes_i \gr$ are the naturally associated crossed products. In their case $\gr$ is an \'etale groupoid and $H$ is a Lie group. Already in their paper (Lemmas 2.1 and 2.2) they proved some very interesting functoriality  properties.  The Hilsum-Skandalis construction can be generalized for any groupoid cocycle between locally compact groupoids as shown by Le Gall in \cite{LeGall99}. Indeed, Le Gall gave  in his paper a precise definition for groupoid equivariant K-theory and constructs for every groupoid cocycle
$$\gr\stackrel{i}{--->}\hr$$ a descent morphism
$$i^*:KK_\hr(A,B)\longrightarrow KK_\gr(i^*A,i^*B)$$
for every $A,B$ $\hr$-algebras, and where $i^*A$ (resp. $i^*B$) is the naturally associated algebra in which $\gr$ acts via the cocycle $i$. The main result in \cite{LeGall99}, Theorem 7.2, states the functoriality and naturality with respect to the Kasparov product of the descent construction\footnote{The Kasparov descent morphisms are a particular case of Le Gall's construction, theorem 7.6 in \cite{LeGall99}.}.
To see how Hilsum-Skandalis construction is contained in Le Gall's one can consider the morphism 
$$KK_\gr(i^*A,i^*B)\stackrel{p^*}{\longrightarrow} KK(A\rtimes_i \gr,B\rtimes_i \gr)$$ 
associated to the projection\footnote{The inclusion of the units is a projection as a generalized morphism, it correspond to the "quotient" map if one interpret the groupoid as a model for the orbit space.} $p:\go --->\gr$
and then $p^*\circ i^*:KK_\hr(A,B)\to KK(A\rtimes_i \gr,B\rtimes_i \gr)$ is Hilsum-Skandalis morphism (that we can still denote $i^*$) for $\hr$ a Lie group.

We will also need to recall the suspension morphism on KK-theory (or  equivariant $KK$). Given a locally compact groupoid $\gr\rightrightarrows M$, for any $A,B,D$ $\gr-C^*$-algebras there is a suspension map
\begin{equation}
\sigma_{M,D}:KK_{\gr}^{i}(A,B)\longrightarrow KK_{\gr}^{i}(D\otimes_{C_0(M)}A,D\otimes_{C_0(M)}B)
\end{equation}
compatible with the $KK$-product, Theorem 6.4 in \cite{LeGall99}.

\subsection{The twisted equivariant Thom isomorphism}
Let  $\alpha$ be a twisting  on $\gr$ and  $\alpha_0$ be the induced twisting on the unit space  $M$.
Given  a $\gr$-manifold    $P$  and let $E\stackrel{q_E}{\longrightarrow}P$ be a $\gr$-oriented vector bundle over P. There are induced twistings $\pi_P^*\alpha_0$ on $P$ and $q_E^*\pi_P^*\alpha_0$ on $E$. 

In \cite{CW} (see also \cite{Kar08}) the twisted Thom isomorphism was established, it gives an isomorphism 

\begin{equation}\label{CWThom}
\xymatrix{
K_{\alpha_0}(P)\ar[r]^-{\mathscr{T}_{E}^{\alpha_0}}_-\cong&K_{\alpha_0+\o_E}(E)
}
\end{equation}
where $\o_E$ is the orientation twisting (\ref{otwisting}) and with a possible shift on the degree depending on the rank of $E$.

In fact, the isomorphism (\ref{CWThom}) can be explicitly described by the Kasparov product with an invertible KK-element 
\begin{equation}\label{CWKKThom}
\beta_{E}^{\alpha_0}\in KK^*(C^*(P,\alpha_0),C^*(E,\alpha_0+\o_E)).
\end{equation}
In the non-equivariant case we can suppose that the vector bundle $E$ is determined by a groupoid cocycle
$$P\stackrel{O_E}{--->}SO(n)$$

Let $C_\tau(\mathbb{R}^n)$ be the algebra of continuous sections vanishing at infinity of the Clifford bundle of $\mathbb{R}^n$. We consider the Thom element $\beta\in KK_{SO(n)}(\mathbb{C},C_\tau(\mathbb{R}^n))$ constructed by Kasparov (Lemma 4 in \cite{Kasparov95}) and usually called the Dual Dirac element. Then the Hilsum-Skandalis-Le Gall's construction yields an element
$$\beta_E:=O_E^*(\beta)\in KK(C_0(P),C^*(E,\o_E))$$
which corresponds by functoriality and naturality with respect to the product of Le Gall's construction to the Thom isomorphism for not necessarily $Spin^c$-vector bundles. Notice that above we can drop equivariant KK-theory since the groupoid is $P\rightrightarrows P$ which acts trivially.
Now, for taking into account $\alpha_0$ one  has  the following suspension map
\begin{equation}
\sigma_{P,C^*(P,\alpha_0)}:KK(C_0(P),C^*(E,\o_E))\longrightarrow KK(C_0(P)\otimes_{C_0(P)}C^*(P,\alpha_0),C^*(E,\o_E)\otimes_{C_0(P)}C^*(P,\alpha_0)).
\end{equation}
Notice  that 
\[
KK(C_0(P)\otimes_{C_0(P)}C^*(P,\alpha_0),C^*(E,\o_E)\otimes_{C_0(P)}C^*(P,\alpha_0))\cong KK(C^*(P,\alpha_0),C^*(E,\alpha_0+\o_E)))
\]
 thanks to proposition 4.8 in \cite{TXring}. We then finally obtain the twisted Thom element
$$\beta_{E}^{\alpha_0}:=\sigma_{P,C^*(P,\alpha_0)}(\beta_E)\in KK(C^*(P,\alpha_0),C^*(E,\alpha_0+\o_E))$$
which gives a $KK$-description of the Thom isomorphism (\ref{CWThom}) in twisted K-theory.



\vspace{3mm}

{\bf The equivariant case}

 
In the equivariant case, if $N\to M$ is $\gr-$manifold and we consider the twisting $\alpha_0$ induced on $N$, then there is no canonical action of $\gr$ on $C^*(N,\alpha_0)$. It is possible however to modify $N$ (by a Morita equivalent groupoid) such that the action is canonical. This is the subject of Theorem 4.2 in \cite{TXring}. Here we just do a different reading: Let $N$ be $\gr$-manifold with momentum map 
$\pi_N:N\to M$. Take as above a twisting $\alpha$ on $\gr$. There is a groupoid $\tilde{N}\rightrightarrows N'$ Morita equivalent to $N\rightrightarrows N$, admitting an action of $\gr$ together with a strict groupoid morphism 
$$\xymatrix{
\tilde{N}\rtimes \gr \ar[rr]^-{\alpha_{\tilde{N}}}&&PU(H)}$$
and an  explicit Morita equivalence 
$$\tilde{N}\rtimes \gr \stackrel{m_N}{--->} N\rtimes \gr$$
fitting the following commutative diagram of generalized morphisms
\begin{equation}
\xymatrix{
\tilde{N}\rtimes \gr \ar@{.>}[d]_-{m_N}\ar[rr]^-{\alpha_{\tilde{N}}}&&PU(H)\\
N\rtimes \gr \ar@{.>}[rru]_-{\pi_N^*\alpha} &&
}
\end{equation}

In particular, the $S^1$-central extension obtained from $\alpha_{\tilde{N}}$ is of the form
\begin{equation}\label{RMoritaNG}
S^1\to R_{\tilde{N}} \rtimes \gr \to \tilde{N} \rtimes \gr
\end{equation}
where 
$R_{\tilde{N}}$ corresponds to the $S^1$-central extension associated to the twisting $\tilde{\alpha_0}$ on $\tilde{N}$. The extension (\ref{RMoritaNG}) is Morita equivalent to the $S^1$-central extension associated to 
$\pi_N^*\alpha$. As an immediate corollary we get a Morita equivalence (\cite{TXring} corollary 4.6) between the $C^*$-algebras
$$C^*(R_{\tilde{N}})\rtimes \gr \sim C^*(R^{N}_{\alpha})$$
preserving the $\mathbb{Z}$-grading (\ref{gradext}). In particular for degree one we get a Morita equivalence
\begin{equation}\label{MoritaNG}
C^*(\tilde{N},\tilde{\alpha_0})\rtimes\gr \sim C^*(N\rtimes \gr,\alpha).
\end{equation}

Let us come back to the definition of the Thom isomorphism in the equivariant case. 

Now $E$ is a $\gr$-vector bundle over $P$. We assume\footnote{We are only interested in this case in this paper.} that $E$ can be obtained from a cocycle
$$O_E:P\rtimes \gr ---> SO(n),$$
or in other terms $E$ admits a $P\rtimes \gr$-invariant metric.

By Le Gall's descent construction we have a morphism
$$O_E^*:KK_{SO(n)}(\mathbb{C},C_\tau(\mathbb{R}^n))\longrightarrow 
KK_\gr(C^*(\tilde{P}), C^*(\tilde{E},\widetilde{\o_E})),$$
where  $\widetilde{\o_E}$  is  defined by  the equivariant orientation twisting $\o_E$   (\ref{otwisting}) associated to $E$.

Next, we consider the suspension map
\begin{equation}\label{A.8}
\sigma_{M,C^*(\tilde{P},\tilde{\alpha_0})}:KK_\gr(C^*(\tilde{P}), C^*(\tilde{E},\widetilde{\o_E}))\to KK_\gr(C^*(\tilde{P})\otimes_{C_0(M)}C^*(\tilde{P},\tilde{\alpha_0}), C^*(\tilde{E},\widetilde{\o_E})\otimes_{C_0(M)}C^*(\tilde{P},\tilde{\alpha_0}))
\end{equation}
and again by Proposition 4.8 in \cite{TXring} we have  canonical  isomorphisms
\begin{center}
$C^*(\tilde{P})\otimes_{C_0(M)}C^*(\tilde{P},\tilde{\alpha_0})\cong C^*(\tilde{P},\tilde{\alpha_0})$ and $C^*(\tilde{E},\widetilde{\o_E})\otimes_{C_0(M)}C^*(\tilde{P},\tilde{\alpha_0})\cong C^*(\tilde{E},\tilde{\alpha_0}+\widetilde{\o_E})$
\end{center}
and hence $\sigma_{M,C^*(\tilde{P},\tilde{\alpha_0})}$ can be considered to take values on $KK_\gr(C^*(\tilde{P},\tilde{\alpha_0}),C^*(\tilde{E},\tilde{\alpha_0}+\widetilde{\o_E}))$. Next we can apply the descent functor to get to $KK(C^*(\tilde{P},\tilde{\alpha_0})\rtimes \gr,C^*(\tilde{E},\tilde{\alpha_0}+\widetilde{\o_E})\rtimes \gr)$ and finally we can use the Morita equivalence (\ref{MoritaNG}) to obtain a
canonical  isomorphism
$$KK(C^*(\tilde{P},\tilde{\alpha_0})\rtimes \gr,C^*(\tilde{E},\tilde{\alpha_0}+\widetilde{\o_E})\rtimes \gr)\cong  KK(C^*(P\rtimes \gr, \alpha),C^*(E\rtimes \gr, \alpha+\o_E)).$$

We have a twisted equivariant Thom element
$$\beta_{E}^{\gr,\alpha}\in KK(C^*(P\rtimes \gr, \alpha),C^*(E\rtimes \gr, \alpha+\o_E)),$$
obtained from $\beta_n \in KK_{SO(n)}(\mathbb{C},C_\tau(\mathbb{R}^n))$ under the 
suspension map  (\ref{A.8}) and the above canonical isomorphisms.


\begin{definition}[Equivariant twisted Thom isomorphism]
We can consider the $K$-theory isomorphism:
\begin{equation}
\xymatrix{
K_{\alpha}^\gr(P)\ar[r]^-{\mathscr{T}_{E}^{\gr,\alpha}}_-{\cong}&
K_{\alpha+\o_E}^\gr(E)
}
\end{equation}
associated to the twisted  equivariant Thom element constructed above, more explicitly,
$$\mathscr{T}_{E}^{\gr,\alpha}(x):=x\otimes \beta_{E}^{\gr,\alpha},$$
where $\otimes$ stands for the Kasparov product over $C^*(P\rtimes \gr,\alpha)$. We will call the morphism given by the previous equation the $\gr$-equivariant twisted Thom isomorphism.
\end{definition}

\begin{remark}
The fact that is indeed the Thom isomorphism comes from the functoriality of the Hilsum-Skandalis-Le Gall's construction together with the compatibility of the suspension map with the Kasparov's product, theorem 7.2 in \cite{LeGall99}.
\end{remark}

The following proposition states some properties that justify the terminology "Thom isomorphism". Properties 2. and 3. are the analogs of propositions 2.9 and 3.6 in \cite{HS} in our setting.

\begin{proposition}\label{Thomproperties}
For the twisted equivariant Thom isomorphism we have the following three properties:
\begin{enumerate}
\item Let $P$ be a $\gr$-space and let $E\stackrel{}{\longrightarrow}P$ be a $\gr$-oriented vector bundle over P. Suppose  we have $\gr\stackrel{\phi}{--->}\gr'$ a generalized isomorphism. Given  $\alpha':\gr'---> PU(H)$ a twisting, there is an induced commutative diagram of isomorphisms between twisted K-theory groups:
\begin{equation}\label{ThomMoritadiagram}
\xymatrix{
K_{\alpha}^\gr(P)\ar[r]^-{\phi^P_*}_-{\cong}\ar[d]_-{\mathscr{T}_{E}^\gr}&K_{\alpha'}^{\gr'}(\phi(P))\ar[d]^-{\mathscr{T}_{\phi(E)}^{\gr'}}\\ 
K_{\alpha+\o_E}^\gr(E)\ar[r]_-{\phi^E_*}^-{\cong}&K_{\alpha'+\o_{\phi(E)}}^{\gr'}(\phi(E))
}
\end{equation}
where $\alpha:=\alpha'\circ \phi$.
\item[2.] Let $P$ be a $\gr$-manifold and $E_1,E_2$ two oriented $\gr$-vector bundles over $P$. Let  $\pi_1:E_1^*\to P$ be the dual vector vector of $E_1$. We have
\begin{equation}
\mathscr{T}_{\pi_1^*E_2}^\gr\circ \mathscr{F}_1 \circ \mathscr{T}_{E_1}^\gr
=\mathscr{F}_2 \circ \mathscr{T}_{E_1\oplus E_2}^\gr
\end{equation}
where $\mathscr{F}_1$ is the K-theory isomorphism induced from the $C^*$-algebra  Fourier isomorphism\footnote{we recall that in this Fourier transform $\gr$ acts on $E\rightrightarrows P$ (groupoid given by vector bundle structure) for the first factor and on $E_1^*\rightrightarrows E_1^*$ (trivial groupoid) on the second.} (\cite{CaWangAdv} proposition 2.12)
$$C^*(E_1\rtimes G,\alpha+\o_{E_1})\to C^*(E_1^*\rtimes G,\alpha+\o_{E_1}^*)$$
and where $\mathscr{F}_2$ is the K-theory isomorphism induced from the $C^*$-algebra  Fourier isomorphism \footnote{Again,$ E_1\oplus E_2\rightrightarrows P$ as vector bundle groupoid and $\pi_1E_2 \rightrightarrows E_1^*$.}
$$C^*((E_1\oplus E_2)\rtimes G,\alpha+\o_{E_1\oplus E_2})\to C^*(\pi_1^*E_2\rtimes G,\alpha+\o_{E_1\oplus E_2}^*)$$

\item[3.]  Let $P$ be a $\gr$-manifold, $E$ an oriented $\gr$-vector bundle over $P$ and $E'$ an oriented $\gr$-vector bundle over $P$ together with a $\gr$-vector bundle $E\to E'$ morphism, we have
\begin{equation}
 \mathscr{T}_{E'}^{E\rtimes\gr}\circ \mathscr{T}_{E}^\gr=\sigma_{E'\rtimes E}^\gr\circ\mathscr{T}_{E\oplus E'}^\gr
\end{equation}
where $\sigma\in KK(((E\oplus E')\rtimes \gr, \alpha+ \o_{E'\oplus E}),(E'\rtimes E)\rtimes \gr,\alpha))$ is the deformation index associated to the deformation groupoid
$$\gr_{E'\rtimes E}:=(E\oplus E')\rtimes \gr\bigsqcup (E'\rtimes E)\rtimes \gr \times (0,1]$$
which can be obtained as the semidirect product of the tangent groupoid of $(E'\rtimes E)$ by the action of $\gr$.
\end{enumerate}
\end{proposition}


\begin{proof}
Properties 1. and 2. follow immediately from functoriality of Le Gall's descent functors  and its naturality with respect to Kasparov products , theorem 7.2 in \cite{LeGall99},  together with the compatibility of the suspension map with the Kasparov's product.

The proof of property 3. is essentiality the same as the proof 3.6 in \cite{HS}, that is, one observes that the tangent groupoid of $E'\rtimes E$, $\mathbb{T}_{E'\rtimes E}$, is a $\gr\times [0,1]$-vector bundle over $P\times [0,1]$, and one can then consider its twisted Thom isomorphism. We have the following diagram\footnote{remember that in our notation $\alpha$ stands for the given twisting on $\gr$, and that we keep denoting by $\alpha$ all the canonically induced twistings from it.} 
\[
\xymatrix{
K^*(P\rtimes \gr,\alpha)\ar[rr]^-{\mathscr{T}_E^\gr}&&K^*(E\rtimes \gr,\alpha+\o_E)\ar[rr]^-{\mathscr{T}_{E'}^{E\rtimes\gr}}&&K^*((E'\rtimes E)\rtimes \gr,\alpha+\o_{E'\oplus E})\\
K^*((P\rtimes \gr)\times [0,1],\alpha)\ar[d]_-{e_0}^-{\cong}\ar[u]_-{e_1}\ar[u]^-{e_1}_-{\cong}\ar[rrrr]_-{\mathscr{T}_{\mathbb{T}_{E'\rtimes E}}^\gr}&&&&
K^*(\mathbb{T}_{E'\rtimes E}\rtimes \gr,\alpha+\o_{E'\oplus E})\ar[d]^-{e_0}_-{\cong}\ar[u]_-{e_1}\\
K^*(P\rtimes \gr,\alpha)\ar[rrrr]_-{\mathscr{T}_{E'\oplus E}^{\gr}}&&&&K^*((E'\oplus E)\rtimes \gr,\alpha+\o_{E'\oplus E})
}
\]
which is commutative. Indeed the top rectangle is commutative by using again Le Gall's theorem, and the bottom one is trivially commutative (deformation indices are compatibles with morphisms induced by evaluations). The result follows from the fact that the left horizontal arrow is the identity.
 
\end{proof}


\bibliographystyle{plain}
\bibliography{./bibliographie}

\begin{thebibliography}{10}

\bibitem{Ab}
Herbert Abels.
\newblock Parallelizability of proper actions, global {$K$}-slices and maximal
  compact subgroups.
\newblock {\em Math. Ann.}, 212:1--19, 1974/75.

\bibitem{ASeg}
Michael Atiyah and Graeme Segal.
\newblock Twisted {$K$}-theory.
\newblock {\em Ukr. Mat. Visn.}, 1(3):287--330, 2004.

\bibitem{BGW}
Guentner-Erik Baum, Paul and Rufus Willett.
\newblock Expanders, exact crossed products, and the baum-connes conjecture.
\newblock {\em Arxiv 1311.2343}, 2013.

\bibitem{BC}
Paul Baum and Alain Connes.
\newblock Geometric {$K$}-theory for {L}ie groups and foliations.
\newblock {\em Enseign. Math. (2)}, 46(1-2):3--42, 2000.

\bibitem{BHS}
Paul Baum, Nigel Higson, and Thomas Schick.
\newblock A geometric description of equivariant {$K$}-homology for proper
  actions.
\newblock In {\em Quanta of maths}, volume~11 of {\em Clay Math. Proc.}, pages
  1--22. Amer. Math. Soc., Providence, RI, 2010.

\bibitem{BOOSW}
Paul Baum, Herv{\'e} Oyono-Oyono, Thomas Schick, and Michael Walter.
\newblock Equivariant geometric {$K$}-homology for compact {L}ie group actions.
\newblock {\em Abh. Math. Semin. Univ. Hambg.}, 80(2):149--173, 2010.

\bibitem{Candel}
Alberto Candel and Lawrence Conlon.
\newblock {\em Foliations. {I}}, volume~23 of {\em Graduate Studies in
  Mathematics}.
\newblock American Mathematical Society, Providence, RI, 2000.

\bibitem{CW}
Alan~L. Carey and Bai-Ling Wang.
\newblock Thom isomorphism and push-forward map in twisted {$K$}-theory.
\newblock {\em J. K-Theory}, 1(2):357--393, 2008.

\bibitem{Ca4}
Paulo Carrillo~Rouse.
\newblock Compactly supported analytic indices for {L}ie groupoids.
\newblock {\em J. K-Theory}, 4(2):223--262, 2009.

\bibitem{CaWangCRAS}
Paulo Carrillo~Rouse and Bai-Ling Wang.
\newblock Twisted index theory for foliations.
\newblock {\em C. R. Math. Acad. Sci. Paris}, 348(23-24):1297--1301, 2010.

\bibitem{CaWangAdv}
Paulo Carrillo~Rouse and Bai-Ling Wang.
\newblock Twisted longitudinal index theorem for foliations and wrong way
  functoriality.
\newblock {\em Adv. Math.}, 226(6):4933--4986, 2011.

\bibitem{CEN}
J{\'e}r{\^o}me Chabert, Siegfried Echterhoff, and Ryszard Nest.
\newblock The {C}onnes-{K}asparov conjecture for almost connected groups and
  for linear {$p$}-adic groups.
\newblock {\em Publ. Math. Inst. Hautes \'Etudes Sci.}, (97):239--278, 2003.

\bibitem{Concg}
Alain Connes.
\newblock {\em Noncommutative geometry}.
\newblock Academic Press Inc., San Diego, CA, 1994.

\bibitem{Deb}
Claire Debord.
\newblock Holonomy groupoids of singular foliations.
\newblock {\em J. Differential Geom.}, 58(3):467--500, 2001.

\bibitem{DLN}
Claire Debord, Jean-Marie Lescure, and Victor Nistor.
\newblock Groupoids and an index theorem for conical pseudo-manifolds.
\newblock {\em J. Reine Angew. Math.}, 628:1--35, 2009.

\bibitem{HF}
Matias Del~Hoyo and Rui Loja~Fernandes.
\newblock Riemannian metrics on lie groupoids.
\newblock {\em Arxiv preprint 1404.5989}.

\bibitem{Ehr}
Charles Ehresmann.
\newblock {\em Cat\'egories et structures}.
\newblock Dunod, Paris, 1965.

\bibitem{FMW}
Igor Fulman, Paul~S. Muhly, and Dana~P. Williams.
\newblock Continuous-trace groupoid crossed products.
\newblock {\em Proc. Amer. Math. Soc.}, 132(3):707--717 (electronic), 2004.

\bibitem{God}
Claude Godbillon.
\newblock {\em Feuilletages}, volume~98 of {\em Progress in Mathematics}.
\newblock Birkh\"auser Verlag, Basel, 1991.
\newblock \'Etudes g\'eom\'etriques. [Geometric studies], With a preface by G.
  Reeb.

\bibitem{HS}
Michel Hilsum and Georges Skandalis.
\newblock Morphismes {$K$}-orient\'es d'espaces de feuilles et fonctorialit\'e
  en th\'eorie de {K}asparov (d'apr\`es une conjecture d'{A}. {C}onnes).
\newblock {\em Ann. Sci. \'Ecole Norm. Sup. (4)}, 20(3):325--390, 1987.

\bibitem{Kar08}
Max Karoubi.
\newblock Twisted {$K$}-theory---old and new.
\newblock In {\em {$K$}-theory and noncommutative geometry}, EMS Ser. Congr.
  Rep., pages 117--149. Eur. Math. Soc., Z\"urich, 2008.

\bibitem{Kasparov95}
G.~G. Kasparov.
\newblock {$K$}-theory, group {$C\sp *$}-algebras, and higher signatures
  (conspectus).
\newblock In {\em Novikov conjectures, index theorems and rigidity, Vol.\ 1
  (Oberwolfach, 1993)}, volume 226 of {\em London Math. Soc. Lecture Note
  Ser.}, pages 101--146. Cambridge Univ. Press, Cambridge, 1995.

\bibitem{KMRW}
Alexander Kumjian, Paul~S. Muhly, Jean~N. Renault, and Dana~P. Williams.
\newblock The {B}rauer group of a locally compact groupoid.
\newblock {\em Amer. J. Math.}, 120(5):901--954, 1998.

\bibitem{Laff12}
Vincent Lafforgue.
\newblock La conjecture de {B}aum-{C}onnes \`a coefficients pour les groupes
  hyperboliques.
\newblock {\em J. Noncommut. Geom.}, 6(1):1--197, 2012.

\bibitem{Las}
Ivan Lassagne.
\newblock {\em Th{\`e}se de {D}octorat {\`a} l'{U}niversit{\'e} de Lorraine},
  2013.

\bibitem{LeGall99}
Pierre-Yves Le~Gall.
\newblock Th\'eorie de {K}asparov \'equivariante et groupo\"\i des. {I}.
\newblock {\em $K$-Theory}, 16(4):361--390, 1999.

\bibitem{Mac}
K.~Mackenzie.
\newblock {\em Lie groupoids and {L}ie algebroids in differential geometry},
  volume 124 of {\em London Mathematical Society Lecture Note Series}.
\newblock Cambridge University Press, Cambridge, 1987.

\bibitem{MM}
I.~Moerdijk and J.~Mr{\v{c}}un.
\newblock {\em Introduction to foliations and {L}ie groupoids}, volume~91 of
  {\em Cambridge Studies in Advanced Mathematics}.
\newblock Cambridge University Press, Cambridge, 2003.

\bibitem{Mr}
Janez Mr{\v{c}}un.
\newblock Functoriality of the bimodule associated to a {H}ilsum-{S}kandalis
  map.
\newblock {\em $K$-Theory}, 18(3):235--253, 1999.

\bibitem{Muh99}
Paul~S. Muhly.
\newblock Bundles over groupoids.
\newblock In {\em Groupoids in analysis, geometry, and physics ({B}oulder,
  {CO}, 1999)}, volume 282 of {\em Contemp. Math.}, pages 67--82. Amer. Math.
  Soc., Providence, RI, 2001.

\bibitem{Pat}
Alan L.~T. Paterson.
\newblock {\em Groupoids, inverse semigroups, and their operator algebras},
  volume 170 of {\em Progress in Mathematics}.
\newblock Birkh\"auser Boston Inc., Boston, MA, 1999.

\bibitem{PPT}
Markus~J. Pflaum, Hessel Posthuma, and Xiang Tang.
\newblock Geometry of orbit spaces of proper {L}ie groupoids.
\newblock {\em J. Reine Angew. Math.}, 694:49--84, 2014.

\bibitem{Ren}
Jean Renault.
\newblock {\em A groupoid approach to {$C\sp{\ast} $}-algebras}, volume 793 of
  {\em Lecture Notes in Mathematics}.
\newblock Springer, Berlin, 1980.

\bibitem{Ren87}
Jean Renault.
\newblock Repr\'esentation des produits crois\'es d'alg\`ebres de groupo\"\i
  des.
\newblock {\em J. Operator Theory}, 18(1):67--97, 1987.

\bibitem{Shim}
Jae-Kwan Shim.
\newblock The invariance of analytic assembly maps under the groupoid
  equivalence.
\newblock {\em J. Math. Kyoto Univ.}, 41(4):809--827, 2001.

\bibitem{Tu2}
Jean-Louis Tu.
\newblock La conjecture de {B}aum-{C}onnes pour les feuilletages moyennables.
\newblock {\em $K$-Theory}, 17(3):215--264, 1999.

\bibitem{Tu1}
Jean~Louis Tu.
\newblock La conjecture de {N}ovikov pour les feuilletages hyperboliques.
\newblock {\em $K$-Theory}, 16(2):129--184, 1999.

\bibitem{Tu3}
Jean-Louis Tu.
\newblock The {B}aum-{C}onnes conjecture for groupoids.
\newblock In {\em $C\sp *$-algebras (M\"unster, 1999)}, pages 227--242.
  Springer, Berlin, 2000.

\bibitem{Tu4}
Jean-Louis Tu.
\newblock Habilitation \`a diriger des recherches en math\'ematiques.
\newblock 2005.

\bibitem{TXChern}
Jean-Louis Tu and Ping Xu.
\newblock Chern character for twisted {$K$}-theory of orbifolds.
\newblock {\em Adv. Math.}, 207(2):455--483, 2006.

\bibitem{TXring}
Jean-Louis Tu and Ping Xu.
\newblock The ring structure for equivariant twisted {$K$}-theory.
\newblock {\em J. Reine Angew. Math.}, 635:97--148, 2009.

\bibitem{TXL}
Jean-Louis Tu, Ping Xu, and Camille Laurent-Gengoux.
\newblock Twisted {$K$}-theory of differentiable stacks.
\newblock {\em Ann. Sci. \'Ecole Norm. Sup. (4)}, 37(6):841--910, 2004.

\bibitem{W08}
Bai-Ling Wang.
\newblock Geometric cycles, index theory and twisted {$K$}-homology.
\newblock {\em J. Noncommut. Geom.}, 2(4):497--552, 2008.

\bibitem{Win}
H.~E. Winkelnkemper.
\newblock The graph of a foliation.
\newblock {\em Ann. Global Anal. Geom.}, 1(3):51--75, 1983.

\end{thebibliography}

\end{document}